\documentclass[12pt,a4paper]{article}

\usepackage[utf8]{inputenc}
\usepackage[english]{babel}
\usepackage[top=1in,bottom=1in,left=1in,right=1in]{geometry}
\usepackage{bm}
\usepackage{graphicx}
\usepackage{float}
\usepackage{booktabs}
\usepackage{array}
\usepackage{ragged2e}
\usepackage{tabularx}
\newcolumntype{Y}{>{\RaggedRight\arraybackslash}X}
\usepackage{subcaption}
\usepackage{amssymb}
\usepackage{lscape}
\usepackage{placeins}

\usepackage{indentfirst}
\usepackage[affil-it]{authblk}
\usepackage{titlesec}
\usepackage{csquotes}
\usepackage{enumerate}
\usepackage{enumitem}
\usepackage{lineno}

\usepackage{amsmath,amsthm}

\theoremstyle{plain} 
\newtheorem{theorem}{Theorem}[section]
\newtheorem{lemma}[theorem]{Lemma}
\newtheorem{proposition}[theorem]{Proposition}

\theoremstyle{definition} 

\theoremstyle{remark} 

\usepackage{hyperref}
\hypersetup{
    colorlinks=true,
    linkcolor=blue,
    filecolor=magenta,      
    urlcolor=cyan,
    citecolor=blue, 
    }

\title{$LCT–SIND^3y$: Sparse Identification of Nonlinear Distributed-Delay Dynamics via the Linear Chain Trick}

\author{Mohammed Alanazi$^{1}$ \qquad Majid Bani-Yaghoub$^{1}$\footnote{Corresponding author. Email: baniyaghoubm@umkc.edu}}
\affil{\small $^{1}$ Division of Computing, Analytics \& Mathematics, School of Science and Engineering, University of Missouri–Kansas City, 5100 Rockhill Rd., Kansas City, Missouri 64110, USA}
\date{\vspace{-5ex}}

\begin{document}
\maketitle

\begin{abstract}
The Sparse Identification of Nonlinear Dynamics (SINDy) framework has been frequently used to discover parsimonious differential equations governing natural and physical systems.  This includes recent extensions to SINDy that enable the recovery of discrete delay differential equations, where delay terms are represented explicitly in the candidate library. However, such formulations cannot capture the distributed delays that naturally arise in biological, physical, and engineering systems. In the present work, we extend SINDy to identify distributed-delay differential equations by incorporating the Linear Chain Trick (LCT), which provides a finite-dimensional ordinary differential equation representing the distributed memory effects. Hence, SINDy can operate in an augmented state space using conventional sparse regression while preserving a clear interpretation of delayed influences via the chain trick. From time-series data, the proposed method jointly infers the governing equations, the mean delay, and the dispersion of the underlying delay distribution. We numerically verify the method on several models with distributed delay, including the logistic growth model and a Hes1--mRNA gene regulatory network model. We show that the proposed method accurately reconstructs distributed delay dynamics, remains robust under noise and sparse sampling, and provides a transparent, data-driven approach for discovering nonlinear systems with distributed-delay.

\end{abstract}



\newpage
\section{Introduction}
There have been intensive activities in developing and applying data-driven modeling techniques in the past decade. Methods such as Sparse Identification of Nonlinear Dynamics (SINDy) and Dynamic Mode Decomposition (DMD), along with many of their variants, have become essential tools for discovering governing models associated with natural and physical systems \cite{brunton2016discovering,kutz2016dynamic,menta2025uncovering, schmid2010dmd,schmid2010dynamic}. On the other hand, machine learning models, ranging from kernel methods and deep neural networks to neural ODEs \cite{arjmand2025comparative,chen2018neural, lecun2015deep, scholkopf2002learning, soysal2022machine}, offer flexible, high-accuracy predictors, but often lack interpretability. In addition, many real systems are intrinsically nonlinear and may involve time delays which are not tractable by machine learning models. Among the equation-discovery approaches, SINDy framework \cite{brunton2016sindy} learns explicit governing equations directly from measured state trajectories. SINDy was originally introduced for ordinary differential equations and later was  extended partial-differential settings. It operates by pairing time series data with a user-specified library of candidate models. Sparse identification is then used to select a small subset of these
candidates and estimate their coefficients, resulting in interpretable models whose terms align with underlying mechanisms.  \\
\indent SINDy has been extended beyond ordinary differential equations to several classes of systems, including partial differential equations \cite{rudy2017pdefind}, stochastic differential equations \cite{boninsegna2018sde}, and stochastic PDEs \cite{mathpati2024vbspde}. More recently, a growing literature adapts SINDy to discrete delay differential equations (DDEs) using diverse strategies, for example, introducing explicit delayed terms with sparse selection, learning parameterized delay dictionaries, or approximating DDEs via pseudospectral collocation before identification \cite{kopeczi-bocz2024sindy-delay,sandoz2023sindy-dde,wu2023parameterizedDDE,bozzo2024psc-sindy,pecile2025dde}. Despite these advances, there is still a need to study this problem.\\
\indent Köpeczi-Bócz et al.\ \cite{kopeczi-bocz2024sindy-delay} augment the SINDy library with a fixed set of preselected lagged regressors (e.g., $x(t-\tau_j)$)
and apply sequential thresholded least squares with a $\lambda$–sweep; delays are inferred within a single fit as those lagged terms that remain active in a
low-error sparse model. In contrast, Sandoz et al.\ \cite{sandoz2023sindy-dde} assume a single delay and perform a bi-level search: for each candidate $\tau$
they build a lagged library and fit a sparse model via greedy backward elimination, then select the best $\tau$ by minimizing a trajectory reconstruction error across candidates. Building on delay-augmented SINDy, Pecile et al.\ \cite{pecile2025dde} cast delay selection as a continuous optimization problem solved via Bayesian optimization: each iteration proposes a delay, forms off-grid lagged regressors by interpolation, refits a sparse model, and scores its data fit to guide the next query.
To mitigate the library blow-up that occurs when many lagged terms are enumerated, Wu~\cite{wu2023parameterizedDDE} proposes a parameterized dictionary in which delay
values are treated as continuous variables and estimated off-grid. Model discovery is posed as a mixed-integer nonlinear program that jointly selects a sparse set of active atoms by binary variables and fits their continuous parameters (including delays), yielding compact models without enumerating a dense grid of candidate lags. Bozzo et al. \cite{bozzo2024psc-sindy} recast delay differential equations as a finite-dimensional ODE via Chebyshev pseudospectral collocation: delayed states are evaluated at collocation nodes and off-grid delays are obtained by interpolation; SINDy is then applied to this surrogate to obtain a sparse model, leveraging mature ODE–SINDy tooling while avoiding the enumeration of large libraries of lagged regressors.\\
\indent The methods mentioned above are only applicable to discrete delays either by enumerating lagged regressors, optimizing a single lag value, or working in collocation coordinates, so they succeed primarily when any underlying distributed delay is sufficiently narrow to appear effectively discrete \cite{kopeczi-bocz2024sindy-delay,sandoz2023sindy-dde,wu2023parameterizedDDE,bozzo2024psc-sindy,pecile2025dde}. Because delays in many biological and engineered systems are distributed, and sometimes broad, discrete-delay SINDy variants can misidentify the model structure. To address this gap, we propose $LCT-SIND^3y$, which leverages the LCT to obtain from the data an augmented state that represents a Gamma/Erlang distributed memory of the dynamics. We then append the terminal chain variable (the distributed-delay state) to the SINDy library, enabling sparse discovery of distributed-delay differential equations while retaining explicit, interpretable right-hand sides. Our method recovers distributed–delay differential equations and the delay term via a bilevel search; even with sparse, noisy data the reconstructed delayed state closely matches the ground truth, giving SINDy a decisive advantage in identifying the correct governing system. To the best of our knowledge, this is the first SINDy framework to directly identify distributed-delay structure via an LCT-based augmentation, providing mechanistic interpretability and improved fidelity when delays are broadly distributed.\\
\indent The rest of this paper is organized as follows. Section \ref{SINDy_Preliminaries} provides the original SINDy framework as preliminaries. Section~\ref{sec:lct-sindy} introduces the  proposed LCT–SINDy method. Subsection \ref{Dis_delays_and} provides details of the Linear Chain Trick, showing how a distributed-delay term (history integral) can be represented by a linear ODE chain of auxiliary variables \cite{bani2017analysis,bani2015oscillatory, bani2020delay, bani2015understanding}. subsection~\ref{Meth_cons} presents the methodological construction of $LCT-SIND^3y$, explaining how the LCT chain is built from data and incorporated into the identification procedure. In Section~\ref{sec4}, we apply the method to a distributed-delay DDE of the Hes1–mRNA system as an illustrative case study, showing when discrete-delay approaches fail and how our approach recovers the correct dynamics when the delay is distributed or nearly discrete. We also show numerically that the reconstructed delayed state achieves low relative error under sparse and noisy sampling. We also present a fully noisy, end-to-end experiment demonstrating that $LCT-SIND^3y$ successfully identifies the governing equations and yields a delay estimate close to the ground truth. Section \ref{Disc} provides a discussion of the main results and the proposed algorithm for distributed delay model discovery.


\begin{table}[H]
\centering
\footnotesize
\caption{Representative SINDy-based approaches for delay differential equations (DDEs).}
\vspace{-.3cm}
\label{tab:sindy_dde_methods}
\setlength{\tabcolsep}{6pt}
\renewcommand{\arraystretch}{1.15}
\begin{tabularx}{\textwidth}{>{\RaggedRight\arraybackslash}p{0.34\textwidth} Y}
\toprule
\textbf{Authors and Year} & \textbf{Core idea of the numerical algorithm} \\
\midrule

Sandoz \emph{et~al.} (2023)~\cite{sandoz2023sindy-dde}
& Augment the SINDy library with lagged regressors, use greedy backward elimination, and select the delay via trajectory–error over a delay grid. \\

K\"opeczi-B\'ocz \emph{et~al.} (2023)~\cite{kopeczi-bocz2024sindy-delay}
& Augments the SINDy library with lagged regressors \(x(t-\tau_j)\) at pre-specified candidate delays; applies STLSQ (\(\lambda\)-sweep) to promote sparsity and select the active delayed terms  \\

Wu (2023)~\cite{wu2023parameterizedDDE}
& Build a parameterized dictionary with atoms encoding continuous delay/shape; jointly learn coefficients and delay parameters via nonconvex optimization. \\

Bozzo, Breda \& Tanveer (2024)~\cite{bozzo2024psc-sindy}
& Replace the DDE with a pseudospectral–collocation surrogate ODE and apply SINDy; infer off-grid delays through the collocation mapping. \\

Pecile \emph{et~al.} (2025)~\cite{pecile2025dde}
& Explore the continuous delay space with Bayesian optimization; at each query, refit a sparse model and score it to guide the search. \\
\bottomrule
\end{tabularx}
\end{table}

\section{SINDy Preliminaries}
\label{SINDy_Preliminaries}
We briefly review the Sparse Identification of Nonlinear Dynamics (SINDy) framework~\cite{brunton2016sindy}. We aim to recover the right–hand side of
\begin{equation}
\dot{\bm{X}}(t)=\bm{f}(\bm{X}(t)), \qquad \bm{X}(t)\in\mathbb{R}^d.
\label{eq:ct_system}
\end{equation}
Collect a time history of the state and either measure $\dot{\bm{X}}(t)$ or approximate it from $\bm{X}(t)$ at times $t_1,\ldots,t_m$. Arrange the samples into
\[
X=
\begin{bmatrix}
\bm{x}^{\!\top}(t_1)\\
\bm{x}^{\!\top}(t_2)\\
\vdots\\
\bm{x}^{\!\top}(t_m)
\end{bmatrix}
=
\begin{bmatrix}
x_1(t_1) & x_2(t_1) & \cdots & x_d(t_1)\\
x_1(t_2) & x_2(t_2) & \cdots & x_d(t_2)\\
\vdots   & \vdots   & \ddots & \vdots\\
x_1(t_m) & x_2(t_m) & \cdots & x_d(t_m)
\end{bmatrix}\!\in\mathbb{R}^{m\times d},\qquad
\dot X=
\begin{bmatrix}
\dot{\bm{x}}^{\!\top}(t_1)\\
\dot{\bm{x}}^{\!\top}(t_2)\\
\vdots\\
\dot{\bm{x}}^{\!\top}(t_m)
\end{bmatrix}\!\in\mathbb{R}^{m\times d}.
\]

Next, we construct a library $\Theta(X)$ of candidate nonlinear functions of the columns of $X$. These candidate functions can be any terms that are expected to best describe the underlying dynamics of the system.
  For example,
\begin{equation}
\Theta(X)=\Big[\,\mathbf{1}\ \ \ X\ \ \ X^{P2}\ \ \ X^{P3}\ \ \ \cdots\ \ \ \sin(X)\ \ \cos(X)\ \ \cdots\,\Big]\in\mathbb{R}^{m\times p},
\label{eq:Theta_block}
\end{equation}
where higher–order polynomial blocks $X^{P2},X^{P3},\ldots$ collect all monomials of the indicated degree. Concretely,
\begin{equation}
X^{P2}=
\begin{bmatrix}
x_1^2(t_1) & x_1(t_1)x_2(t_1) & \cdots & x_d^2(t_1)\\
x_1^2(t_2) & x_1(t_2)x_2(t_2) & \cdots & x_d^2(t_2)\\
\vdots     & \vdots           & \ddots & \vdots\\
x_1^2(t_m) & x_1(t_m)x_2(t_m) & \cdots & x_d^2(t_m)
\end{bmatrix}.
\end{equation}

Once the time derivatives $\dot X$ and the library matrix $\Theta(X)$ have been computed from the data, the next step is to identify which candidate functions are active in the dynamics. We assume that only a small subset of the library terms is needed to represent the right-hand side of equation~\ref{eq:ct_system}, and we therefore seek a sparse coefficient matrix $\Xi$ such that
\begin{equation}
  \dot X \;\approx\; \Theta(X)\,\Xi, 
  \qquad
  \Xi = \big[\,\boldsymbol{\xi}_1\ \boldsymbol{\xi}_2\ \cdots\ \boldsymbol{\xi}_d\,\big] 
  \in \mathbb{R}^{p\times d},
  \label{eq:sindy_matrix}
\end{equation}
where the column vector $\boldsymbol{\xi}_k$ contains the coefficients for the $k$th state variable.

To identify these sparse coefficient vectors, we solve a separate sparse regression problem for each state. For $k = 1,\dots,d$ we consider
\begin{equation}
  \boldsymbol{\xi}_k \in 
  \arg\min_{\beta \in \mathbb{R}^p}
  \big\|\Theta(X)\,\beta - \dot X_{(:,k)}\big\|_2^2
  \quad \text{subject to $\beta$ being sparse},
  \label{eq:sindy_sparse_regression}
\end{equation}
that is, we seek a least–squares fit with as few nonzero coefficients as possible.

There are many algorithms for solving such sparse regression problems. In the original SINDy formulation of Brunton et al.\cite{brunton2016sindy}, the coefficients are computed using sequential thresholded least squares (STLSQ). Starting from the unconstrained least–squares solution of \eqref{eq:sindy_sparse_regression}, coefficients whose magnitude is below a threshold $\lambda>0$ are set to zero, defining an active set of terms. The regression is then refit on this reduced set of active columns of $\Theta(X)$, and the thresholding–refitting procedure is iterated until convergence to a sparse coefficient matrix $\Xi$. By sweeping $\lambda$ from small values (yielding dense, potentially overfit models) to larger values (yielding overly sparse, underfit models), one can identify a range of thresholds where the trade–off between sparsity and reconstruction error is balanced; models in this ``elbow'' region typically correspond to the correct governing equations.

\section{$LCT-SIND^3y$ Method}
\label{sec:lct-sindy}

We propose a data–driven framework that combines the Linear Chain Trick (LCT) with SINDy to identify systems with distributed delays from time–series data. The LCT provides a finite–dimensional ODE representation of a broad class of distributed delays, which we then fit using SINDy. This section explains the modeling assumptions, the LCT construction, and our identification pipeline.

\subsection{Distributed delays and Linear Chain Trick}
\label{Dis_delays_and}
Let $x(t)\in\mathbb{R}^d$ denote the state. A general distributed–delay system has the form
\begin{equation}
\dot{x}(t)
\;=\;
f\!\Big(
  x(t),\,
  \int_{0}^{\infty} K(s)\, x(t-s)\,ds
\Big),
\label{eq:dist-delay-general}
\end{equation}
where 
$f : \mathbb{R}^d \times \mathbb{R}^d \to \mathbb{R}^d$ 
captures the instantaneous and history–dependent dynamics, and
$K : \mathbb{R}_{\ge 0} \to \mathbb{R}$ 
is an integrable memory kernel satisfying 
$\int_{0}^{\infty}\!|K(s)|\,ds < \infty$.

For many applications one chooses a gamma (Erlang) family
\begin{equation}
K_{p,a}(s)=\frac{a^{p}\,s^{p-1}e^{-as}}{(p-1)!},\qquad s\ge 0,
\label{eq:gamma-kernel}
\end{equation}
with integer order $p\in\mathbb{N}$ and rate $a>0$. Note that
$\int_{0}^{\infty} K_{p,a}(s)\,ds=1$, so $K_{p,a}$ is a normalized kernel
with mean $\mu=p/a$ and variance $\sigma^2=p/a^2$; thus $p$ controls the
dispersion around the mean delay $\tau =\mu$.

For a discrete delay fix a target delay $\tau>0$ and consider the Erlang sequence with
\[
p=n,\qquad a=\frac{n}{\tau},\qquad n=1,2,\dots
\]
so that $\mathbb{E}[s]=p/a=\tau$ and $\operatorname{Var}(s)=p/a^{2}=\tau^{2}/n\to0$.
Hence $K_{n,n/\tau}$ concentrates its mass at $s=\tau$ as $n\to\infty$. In fact,
for any bounded continuous input $u:\mathbb{R}\to\mathbb{R}$,
\begin{equation}
\int_{0}^{\infty} K_{n,n/\tau}(s)\,u(t-s)\,ds \;\longrightarrow\; u(t-\tau)
\qquad (n\to\infty),
\label{eq:erlang-limit-discrete}
\end{equation}
i.e., the distributed delay converges to a discrete delay at $\tau$.
(See for example \cite{macdonald1989biological,smith2010linearchain}.)

The LCT replaces the distributed delay in \eqref{eq:dist-delay-general} by an exactly equivalent finite–dimensional ODE \cite{bani2017analysis,bani2015oscillatory, bani2020delay, bani2015understanding}. 
For the gamma kernel \(K_{p,a}\) with mean delay \(\tau=p/a\), the LCT asserts that 
\(z_p(t)=\int_{0}^{\infty} K_{p,a}(s)\,u(t-s)\,ds\) coincides with the last state of a \(p\)-stage first–order cascade with common rate \(a\) (i.e., a chain of length \(n=p\) corresponding to delay \(\tau\)). 
Below we recall this construction and its justification, following \cite{macdonald1989biological,smith2010linearchain}.\\


Although the existence, uniqueness, and representation of the chain have been proven in \cite[Prop.~7.3]{smith2010linearchain}, it does not address the lemma stated below. For completeness, we therefore provide its proof here. By \eqref{eq:gamma-kernel} we define \(K_{p,a}\) as the gamma kernel on \(s\ge 0\).
For \(j=1,\dots,p\) we write \(K_{j,a}\) for the member of order \(j\).

\begin{lemma}[cf.~\cite{smith2010linearchain}]
\label{lem:erlang-ivp}
Let \(K_{p,a}\) be given by \eqref{eq:gamma-kernel}. Then, for \(j=1,\dots,p\), the Erlang kernels \(K_{j,a}\) satisfy the initial–value problem
\begin{equation}
\label{eq:erlang-ivp}
\left\{
\begin{aligned}
\frac{d}{ds}K_{1,a}(s) &= -a\,K_{1,a}(s), && s>0,\\
\frac{d}{ds}K_{j,a}(s) &= a\bigl(K_{j-1,a}(s)-K_{j,a}(s)\bigr), && s>0,\quad j=2,\dots,p,\\
K_{1,a}(0) &= a,\\
K_{j,a}(0) &= 0, && j=2,\dots,p.
\end{aligned}
\right.
\end{equation}
\end{lemma}

\begin{proof}
For \(j=1\), \(K_{1,a}(s)=a e^{-a s}\) for \(s> 0\), hence
\[
\frac{d}{ds}K_{1,a}(s)=-a^2 e^{-a s}=-a\,K_{1,a}(s),
\qquad K_{1,a}(0)=a.
\]
For \(j\ge 2\) and \(s>0\),
\[
\frac{d}{ds}K_{j,a}(s)
= \frac{a^{j}}{(j-1)!}\,\frac{d}{ds}\!\big(s^{j-1}e^{-a s}\big)
= \frac{a^{j}}{(j-1)!}\Big((j-1)s^{j-2}-a s^{j-1}\Big)e^{-a s}.
\]

\[
a\big(K_{j-1,a}(s)-K_{j,a}(s)\big)
= a\left(\frac{a^{j-1}s^{j-2}e^{-a s}}{(j-2)!}-\frac{a^{j}s^{j-1}e^{-a s}}{(j-1)!}\right)
= \frac{a^{j}}{(j-1)!}\Big((j-1)s^{j-2}-a s^{j-1}\Big)e^{-a s},
\]
so \(\tfrac{d}{ds}K_{j,a}(s)=a\big(K_{j-1,a}(s)-K_{j,a}(s)\big)\) for \(s>0\).
Finally, \(K_{j,a}(0)=0\) for \(j\ge 2\). This proves \eqref{eq:erlang-ivp}.
\end{proof}

Now assume \(u:(-\infty,T)\to\mathbb{R}\) is bounded and continuous.
For \(j=1,\dots,p\), define the auxiliary states \(z_j:(-\infty,T)\to\mathbb{R}\) by
\begin{equation}
\label{eq:aux-states}
z_j(t)=\int_{0}^{\infty} K_{j,a}(s)\,u(t-s)\,ds
      =\int_{-\infty}^{t} u(r)\,K_{j,a}(t-r)\,dr .
\end{equation}

For \(j\ge 2\) (so \(K_{j,a}(0)=0\)), differentiating \eqref{eq:aux-states} yields
\begin{align}
\dot z_j(t)
  &= \frac{d}{dt}\int_{-\infty}^{t} u(\eta)\,K_{j,a}(t-\eta)\,d\eta \\
\intertext{by the Leibniz rule (justified by dominated convergence since \(u\in C_b\) and \(K_{j,a}\in L^1\))}
  &= u(t)\,K_{j,a}(0)\;+\;\int_{-\infty}^{t} u(\eta)\,\partial_t K_{j,a}(t-\eta)\,d\eta .
\label{eq:leibniz-step}
\end{align}
By Lemma~\ref{lem:erlang-ivp},
\[
\dot z_j(t)
= u(t)\,K_{j,a}(0)+a\!\int_{-\infty}^{t} u(\eta)\,[K_{j-1,a}-K_{j,a}](t-\eta)\,d\eta
= u(t)\,K_{j,a}(0)+a\,[z_{j-1}(t)-z_j(t)] .
\]
For \(j\ge2\), \(K_{j,a}(0)=0\), hence \(\dot z_j(t)=a\,[z_{j-1}(t)-z_j(t)]\).
For \(j=1\), \(K_{1,a}(0)=a\) and \(\partial_t K_{1,a}(t-\eta)=-aK_{1,a}(t-\eta)\), so
\[
\dot z_1(t)=a\,[u(t)-z_1(t)].
\]
Setting \(z_0(t):=u(t)\), we obtain the compact form
\begin{equation}
\label{eq:lct-chain-deriv}
\dot z_j(t)=a\,[z_{j-1}(t)-z_j(t)],\qquad j=1,\dots,p,
\end{equation}
on \((-\infty,T)\).

If \(T=\infty\) and \(u\in C_b(\mathbb{R})\), we have the following.

\begin{proposition}[LCT: existence–uniqueness and representation; cf.~\cite{smith2010linearchain}, Prop.~7.3]
\label{prop:lct-uniq}
Let \(u\in C_b(\mathbb{R})\) and set \(z_0(t):=u(t)\).
Then the chain \eqref{eq:lct-chain-deriv} admits a unique solution \((z_1,\dots,z_p)\) that is bounded on \(\mathbb{R}\).
Moreover, this solution is given by
\begin{equation}
\label{eq:lct-repr}
z_j(t)=\int_{0}^{\infty} K_{j,a}(s)\,u(t-s)\,ds
      =\int_{-\infty}^{t} u(\eta)\,K_{j,a}(t-\eta)\,d\eta,\qquad j=1,\dots,p.
\end{equation}
\end{proposition}

\noindent\emph{Proof.} See \cite[Prop.~7.3]{smith2010linearchain}.

\subsection{Methodological Construction}
\label{Meth_cons}

In the previous section, we established that a distributed delay can be theoretically represented using the linear chain trick. In this section, we demonstrate how the linear chain trick is implemented in conjunction with the SINDy framework. 

Let $x(t) = (x_1(t),\ldots,x_d(t))^\top \in \mathbb{R}^d$ denote the state.
For each component $i$ we introduce an LCT chain of order $p_i$ and rate
$a_i$ driven by $u_i(t)=x_i(t)$. Let $z^{(i)}_{p_i}(t)$ denote the last
state of the $i$th chain, and collect these terminal states into the vector
\[
z_p(t) := \big(z^{(1)}_{p_1}(t),\ldots,z^{(d)}_{p_d}(t)\big)^\top \in \mathbb{R}^d.
\]
We augment the SINDy library for the $x$–dynamics with the features
$\{z^{(i)}_{p_i}(t)\}_{i=1}^d$. Since only the last chain states enter the
right-hand side, SINDy identifies models of the form
\begin{equation}
\dot x(t) = f\big(x(t),z_p(t)\big),
\end{equation}
where sparsity selects the relevant delayed channels.

Given time series measurements $\{t_i,x(t_i)\}_{i=1}^{N}$, we first denoise the trajectories and estimate time derivatives. For the denoising step we usethe Savitzky--Golay filter~\cite{savitzky1964smoothing}, implemented in \textsc{Matlab} as \texttt{sgolayfilt} with The filter window is chosen in a problem--dependent manner and scaled with the sampling
interval to span an appropriate physical time horizon (approximately $30$ time units in the Hes1--mRNA example) and the polynomial degree is taken up to $3$, reduced if necessary when the window is too short. After smoothing each state component, time derivatives $\dot{x}(t_i)$ are
computed using a standard finite-difference gradient.

We define a finite grid of candidate distributed delays $\{(\tau_m,p_m)\}$, where $\tau_m$ is the target mean delay and $p_m$ is the length of the corresponding linear chain. For each candidate pair $(\tau_m,p_m)$ we construct the associated linear–chain ODEs (using $a = p_m/\tau_m$) and
drive them with the smoothed time–series data $x(t)$. All chains are initialized consistently from the first observed state $x(t_1)$, and we
integrate them on the data time grid to obtain, for each state variable, a terminal chain state $z_p(t)$ that serves as a surrogate for the distributed delay term. These candidate delayed states are then appended to the SINDy library, and sparse regression selects only a small subset of delayed channels to appear in the learned model.

We estimate the coefficient matrix $\Xi$ in
\begin{equation}
\dot{X} = \Theta\big(X, Z_{\text{last}}\big)\,\Xi,
\label{eq:sparse-model}
\end{equation}
using sequentially thresholded ridge regression (STRidge)~\cite{rudy2017datadriven}.
Let $\Theta = \Theta\!\big(x(t_i), Z_{\text{last}}(t_i)\big)$ denote the library
and $\dot{X}$ the stacked targets, and write
$\Xi = [\,\boldsymbol{\xi}_1,\ldots,\boldsymbol{\xi}_d\,]$ with one coefficient
vector per state component. After column-wise normalization of $\Theta$, STRidge
is applied independently to each component $j=1,\ldots,d$ by iterating
\[
\boldsymbol{\xi}_j \leftarrow
\arg\min_{\boldsymbol{\xi}}\,
\big\|\dot{X}_{(:,j)} - \Theta\,\boldsymbol{\xi}\big\|_2^2
\;+\; \lambda\,\|\boldsymbol{\xi}\|_2^2,
\qquad
(\boldsymbol{\xi}_j)_k \leftarrow 0 \ \text{if}\ \big|(\boldsymbol{\xi}_j)_k\big|<\lambda,
\]
followed by refitting the ridge regression on the active support. We repeat
this ridge–threshold–refit cycle until the support stabilizes for each
$\boldsymbol{\xi}_j$.

For each candidate pair $(\{\tau_m,p_m\})$, the data are first split into training and validation intervals. Model identification is performed using the training data, after which the identified ODE is simulated forward in time.
Its trajectory is then compared to the validation data on a fixed scoring grid, and the mean squared error is computed as
\[
\mathrm{MSE} \;=\; \frac{1}{N d}\sum_{i=1}^{N}
\big\|x_{\mathrm{id}}(t_i)-x(t_i)\big\|^2 .
\]
Model selection is subsequently carried out using the Bayesian Information Criterion. ~\cite{kass1995bayes}
\[
\mathrm{BIC} \;=\; N\ln(\mathrm{MSE}) + k\ln N \;+\; \alpha \sum_{m=1}^{M} p_m,
\]
where $k$ denotes the number of active (nonzero) coefficients in $\Xi$, $d$ is the number of observed state components, and $\alpha>0$ is a mild penalty on the total chain length, introduced to discourage unnecessarily long linear-chain representations.

\subsection{Stability of the Reconstructed Delayed State}
\label{subsec:LCT-stability}

The terminal state of the LCT chain depends on the input trajectory through a convolution with the Erlang kernel. Consequently, the accuracy of the reconstructed delayed state $\tilde z_p(t)$ depends on how well the smoothed
and interpolated data $\tilde u(t)$ approximate the true continuous signal $u(t)$. Since the Erlang kernel is nonnegative and integrates to one, the LCT operator defines a non-expansive linear filter in the $L^\infty$ norm: perturbations in the input propagate to the delayed state without amplification. This implies that the LCT reconstruction is stable with respect to measurement noise, sampling errors, and smoothing artifacts. The following proposition formalizes this stability property. 

\begin{proposition}[Stability of the LCT delayed state to input perturbations]
\label{p1}
Let $u, \tilde u \in C_b(\mathbb{R})$ be two bounded continuous input signals and let $K_{p,a}$ be the Erlang kernel as defined in~\eqref{eq:gamma-kernel}. Define the corresponding terminal LCT states (cf.\ Proposition~\ref{prop:lct-uniq}) by
\[
z_p(t) = \int_0^\infty K_{p,a}(s)\,u(t-s)\,ds, \qquad 
\tilde z_p(t) = \int_0^\infty K_{p,a}(s)\,\tilde u(t-s)\,ds.
\]
Then for every $T>0$,
\[
\| \tilde z_p - z_p \|_{L^\infty(0,T)} \;\le\;
\| \tilde u - u \|_{L^\infty(-\infty,T)}.
\]
\end{proposition}

\begin{proof}
From the convolution representation of the LCT chain
(cf.\ Proposition~\ref{prop:lct-uniq}), we have
\[
z_p(t) = \int_0^\infty K_{p,a}(s)\,u(t-s)\,ds,\qquad
\tilde z_p(t) = \int_0^\infty K_{p,a}(s)\,\tilde u(t-s)\,ds.
\]
Subtracting these expressions gives
\[
\tilde z_p(t) - z_p(t) = \int_0^\infty K_{p,a}(s)\,\bigl(\tilde u(t-s) - u(t-s)\bigr)\,ds.
\]
Using $K_{p,a}(s) \ge 0$ and the triangle inequality, we obtain
\[
\bigl| \tilde z_p(t) - z_p(t) \bigr|
\le
\int_0^\infty K_{p,a}(s)\,
\bigl| \tilde u(t-s) - u(t-s) \bigr|\,ds.
\]
For any $t \in (0,T)$ and $s \ge 0$, we have $t-s \le T$, hence
\[
\bigl| \tilde u(t-s) - u(t-s) \bigr|
\le
\sup_{r \le T}
\bigl| \tilde u(r) - u(r) \bigr|
=
\| \tilde u - u \|_{L^\infty(-\infty,T)}.
\]
Therefore,
\[
\bigl| \tilde z_p(t) - z_p(t) \bigr|
\le
\| \tilde u - u \|_{L^\infty(-\infty,T)}
\int_0^\infty K_{p,a}(s)\,ds.
\]
Since the Erlang kernel integrates to $1$, we obtain
\[
\bigl| \tilde z_p(t) - z_p(t) \bigr|
\le
\| \tilde u - u \|_{L^\infty(-\infty,T)}
\qquad
\forall t \in (0,T).
\]
Taking the supremum over $t \in (0,T)$ yields the desired bound,
\[
\| \tilde z_p - z_p \|_{L^\infty(0,T)}
\le
\| \tilde u - u \|_{L^\infty(-\infty,T)}.
\]
\end{proof}

Proposition~\ref{p1} provides a rigorous guarantee that the LCT reconstruction of the delayed state does not amplify input errors: the worst-case error in $\tilde z_p(t)$ on $[0,T]$ is bounded by the worst-case error in the input signal on $(-\infty,T]$. In practice, experimental measurements of $u(t)$ are noisy and sparsely sampled, so the denoised and interpolated signal
$\tilde u(t)$ only approximates the true state. Since the LCT chain is built from $\tilde u(t)$, it is therefore essential that the terminal state $\tilde z_p(t)$ remains stable with respect to these perturbations. In the next section, we illustrate this behavior numerically and contrast it with numerical differentiation, which is well known to be unstable and to amplify noise in the data.







\section{Numerical Analysis}
\label{sec4}
\subsection{Performance Comparison}
\label{sec4.1}
We present numerical evidence that our $LCT-SIND^3y$ identification framework is more general than discrete-delay SINDy by directly comparing the two approaches. The discrete-delay SINDy baseline follows the same overall procedure as in the previous section, except that the library is augmented with an explicit discrete-lag feature (e.g., $X(t-\tau)$) rather than with the final state of the auxiliary LCT chain, which represents the distributed-delay effect. The $LCT-SIND^3y$ framework searches over an Erlang delay kernel with mean $\tau$ and order $p$ while learning the right–hand–side terms; since the kernel variance is $\tau^{2}/p$, large $p$ produces a narrow, almost point–mass memory, whereas smaller $p$ yields a broader, distributed memory~\cite{smith2010linearchain}. When the data are consistent with an effectively discrete delay the procedure selects a large but finite $p$ that well approximates the discrete case; when a broader memory is needed it selects a smaller $p$, indicating a genuinely distributed delay. When the experimental data arise from a broadly distributed delay, any discrete–delay identification method will inevitably misestimate the delay~$\tau$ and misrepresent the governing dynamics, owing to the structural mismatch between the true and assumed memory kernels. To illustrate these points we use a Hes1–mRNA feedback model and compare the two approaches on synthetic data. In the distributed–delay formulation, transcription is repressed by a smoothed (past–averaged) protein level,
\begin{equation}
\begin{aligned}
\dot{M}(t) &= \frac{\alpha_m}{1 + \left( \dfrac{\displaystyle \int_{0}^{\infty} K(s)\, P(t-s)\, ds}{P_0} \right)^{n}} - M_m\, M(t), \\[6pt]
\dot{P}(t) &= \alpha_p\, M(t) - M_p\, P(t),
\end{aligned}
\label{eq:hes1_distributed}
\end{equation}
where $K(s)$ is a delay kernel, $P_0$ is the Hill threshold and $n$ the Hill exponent \cite{Rateitschak2007IntracellularDelay,Alanazi2025Hes1mRNA}. Hes1 ( P(t) ) is a basic helix–loop–helix transcriptional repressor that down–regulates its own expression, while the Hes1 mRNA ( M(t) ) carries the gene’s information to the cytoplasm for translation into protein \cite{Hirata2002Hes1Oscillation}. Because transcription, translation and intracellular transport take finite time, the feedback is naturally delayed. In our experiments we generate time series from the distributed–delay model above and then identify dynamics using SINDy with a library that includes instantaneous polynomial terms and a lagged Hill feature.We compare two approaches: (i) a conventional SINDy formulation that includes only a discrete delay in its library and searches over the delay value~$\tau$, and (ii) our $LCT-SIND^3y$ framework, which performs a joint search over $(\tau, p)$. The results show that the latter accurately recovers the correct delay profile when the data are distributed, while it naturally selects a large~$p$ when the underlying dynamics are discrete-like. We use \textsc{Matlab} to solve system~\eqref{eq:hes1_distributed} and generate simulated data sampled at $\Delta t = 0.02$. The parameter values are given in Table~\ref{tab:parameters}.

\begin{table}
\centering
\caption{Range and baseline values of Model parameters used in the simulations.}
\vspace{-.3cm}
\begin{tabular}{clcc}
\hline
\textbf{Symbol} & \textbf{Description} & \textbf{Baseline Value (Range)} & \textbf{Refs.} \\
\hline
$\mu_m$ & Rate of mRNA degradation & 0.03 (0.026, 0.03) & \cite{Hirata2002Hes1Oscillation} \\
$\mu_H$ & Rate of protein degradation & 0.03 (0.027, 0.036)& \cite{Hirata2002Hes1Oscillation} \\
$\tau$  & Time delay & 10 - 20 & \cite{monk2003oscillatory} \\
$P_0$   & Repression threshold & 100 (10, 100)& \cite{monk2003oscillatory} \\
$n$     & Hill coefficient & 5 (2, 10)& \cite{monk2003oscillatory} \\
$\alpha_H$ & Hes1 protein production rate & 2 & \cite{sturrock2011spatio} \\
$\alpha_m$ & mRNA production rate & 1 & \cite{sturrock2011spatio} \\
\hline
\end{tabular}
\label{tab:parameters}
\end{table}

We then apply discrete--delay SINDy to identify a delay model from this data. When the true dynamics have a broadly distributed delay (small~$p$), the algorithm fails to recover the correct system: it underfits if we restrict the candidate library to linear terms (Figure~\ref{fig:discreteSINDyH}), and it overfits if we allow a third--order polynomial library.
 
\begin{figure}
\centering
\includegraphics[width=1.00\textwidth]{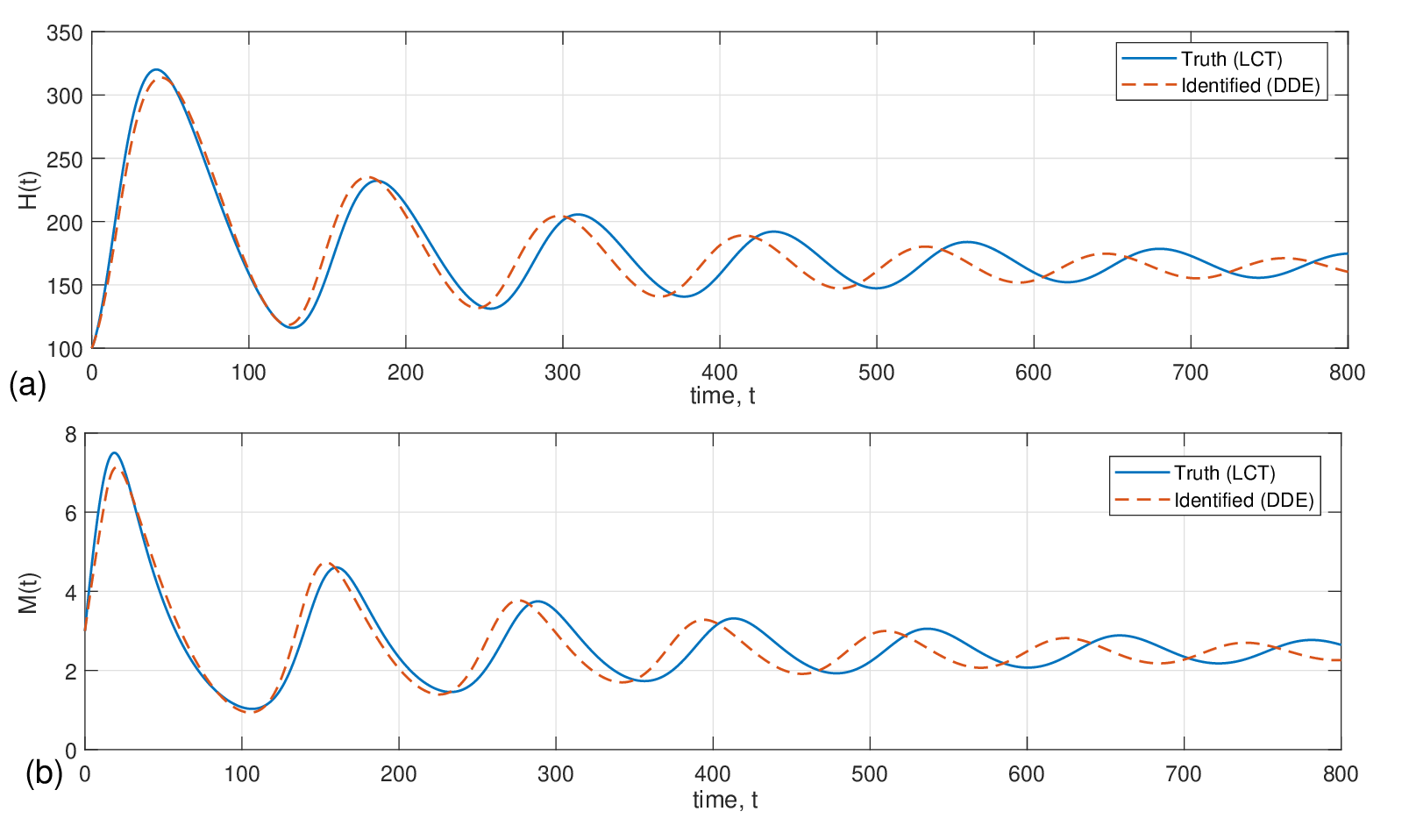} 

\caption{Discrete-delay SINDy identification results. (a) simulation of Hes1 protein. (b) simulation of mRNA. The parameters used to generate the synthetic data are given in Table~\ref{tab:parameters}. We use a broadly distributed delay with $p = 2$.}
\label{fig:discreteSINDyH}
\end{figure}

\begin{table}[h]
\caption{Identified model terms using discrete-delay SINDy. The ground-truth model has $p = 2$ and $\tau = 20$, whereas the identified discrete delay is $\tau^\ast = 13$.}
\label{tab:identified_terms}
\vspace{-.3cm}
\begin{tabular}{lccccccc}
\toprule
\textbf{Eq.} 
    & $1$ 
    & $M$ 
    & $P$ 
    & $\mathrm{Hill}(P_{\tau})$
    & $\mathrm{RSS}$ 
    & $\mathrm{Rel.\ Error}$
    & $\mathrm{BIC}$ \\
\midrule
$\dot{M}$ 
    & $0.019014$
    & $-0.012943$
    & $-0.000243$
    & $0.662368$
    & $6.81\times 10^{3}$
    & $1.41\times 10^{-1}$
    & $-7.08\times 10^{4}$ \\
$\dot{P}$
    & $0$
    & $1.999964$
    & $-0.029999$
    & $0$
    & $6.96\times 10^{6}$
    & $7.27\times 10^{-2}$
    & $2.06\times 10^{5}$ \\
\bottomrule
\end{tabular}

\begin{minipage}{0.95\linewidth}
\footnotesize
\vspace{.3cm}
Total model fit: BIC = $1.36 \times 10^{5}$, 
RSS = $6.96\times 10^{6}$.
\end{minipage}
\end{table}

From Table~\ref{tab:identified_terms}, we observe that the discovered equation for Hes1 protein is correct, whereas the equation identified for mRNA is not. The identified model also underestimates the delay, inferring a delay of $13$, while the true delay is $20$ This discrepancy arises because the candidate library includes only discrete delays, while the data were generated from a system with a distributed delay. In this case, the polynomial terms in the library were restricted to be linear, which led to underfitting. We allow the library to include terms up to third order, which results in a very close fit to the data but produces identified equations that differ from the true system and contain biologically implausible high-order nonlinearities, indicating overfitting. The discrete SINDy method can recover the correct model only when $p$ is large.

\begin{figure}
\centering
\includegraphics[width=0.8\textwidth]{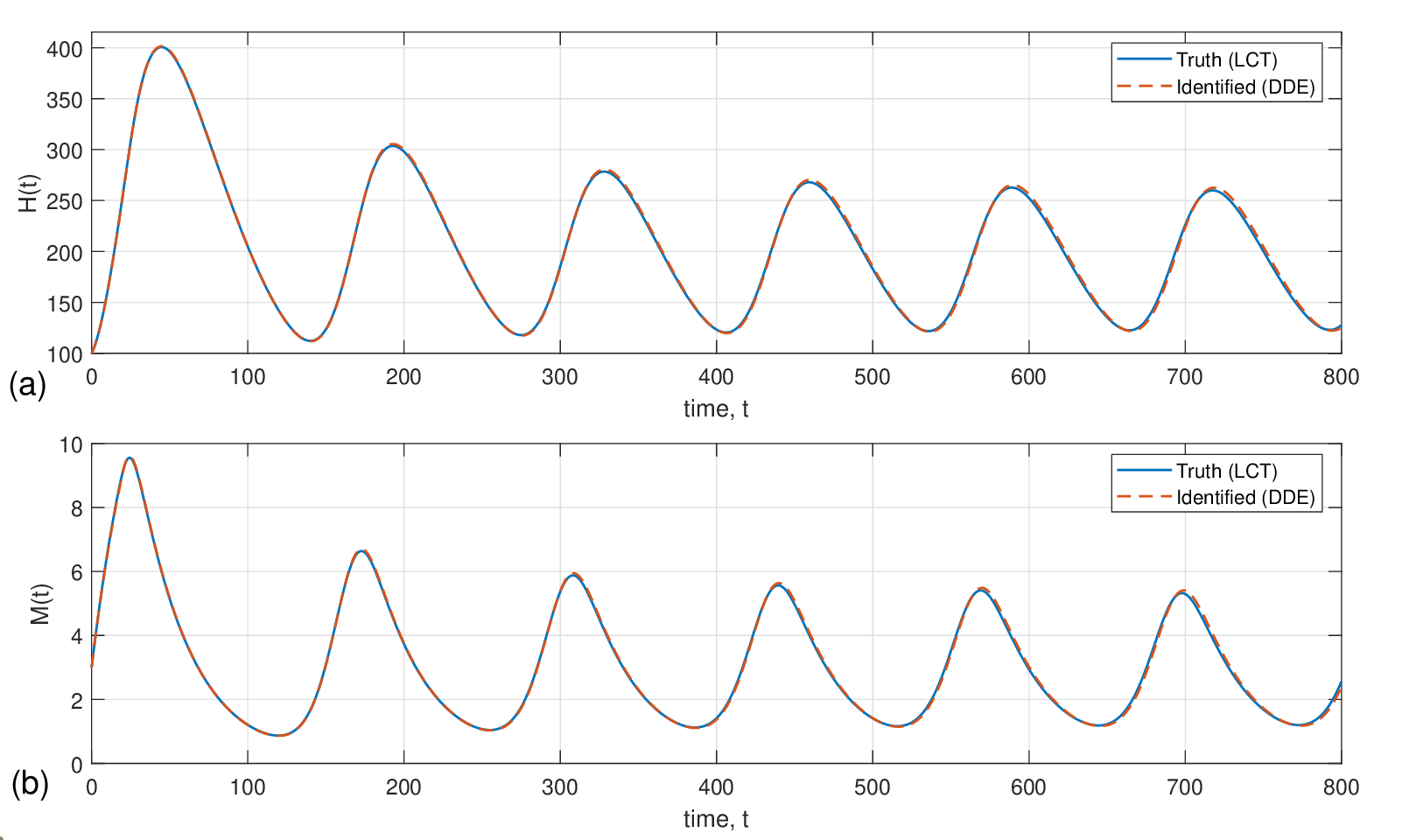} 
\caption{Discrete-delay SINDy identification results. (a) simulation of Hes1 protein. (b) simulation of mRNA. We use the same parameter values as in Figure~\ref{fig:discreteSINDyH}, except with $p = 100$.}
\label{fig:discreteSINDy2}

\end{figure}

\begin{table}[h]
\caption{Identified model terms using discrete-delay SINDy.  The ground-truth model has $p=100$ and $\tau=20$, and the identified discrete delay is $\tau^\ast=20$.}
\label{tab:identified_terms_p100}
\vspace{-0.3cm}
\begin{tabular}{lccccccc}
\toprule
\textbf{Eq.} 
    & $1$ 
    & $M$ 
    & $P$ 
    & $\mathrm{Hill}(P_{\tau})$
    & $\mathrm{RSS}$ 
    & Rel. error
    & $\mathrm{BIC}$ \\
\midrule
$\dot{M}$ 
    & $0$
    & $-0.030082$
    & $0$
    & $0.992949$
    & $1.45 \times 10^{2}$ 
    & $1.662 \times 10^{-2}$
    & $-2.25 \times 10^{5}$ \\
$\dot{P}$
    & $0$
    & $1.999980$
    & $-0.030000$
    & $0$
    & $1.62 \times 10^{5}$ 
    & $9.33 \times 10^{-3}$
    & $5.6 \times 10^{4}$\\
\bottomrule
\end{tabular}
\begin{minipage}{0.95\linewidth}
\footnotesize
\vspace{.3cm}
Total model fit: BIC = $-1.69 \times 10^{5}$, RSS = $1.62 \times 10^{5}$.
\end{minipage}
\end{table}

\begin{figure}
\centering

\includegraphics[width=1.00\textwidth]{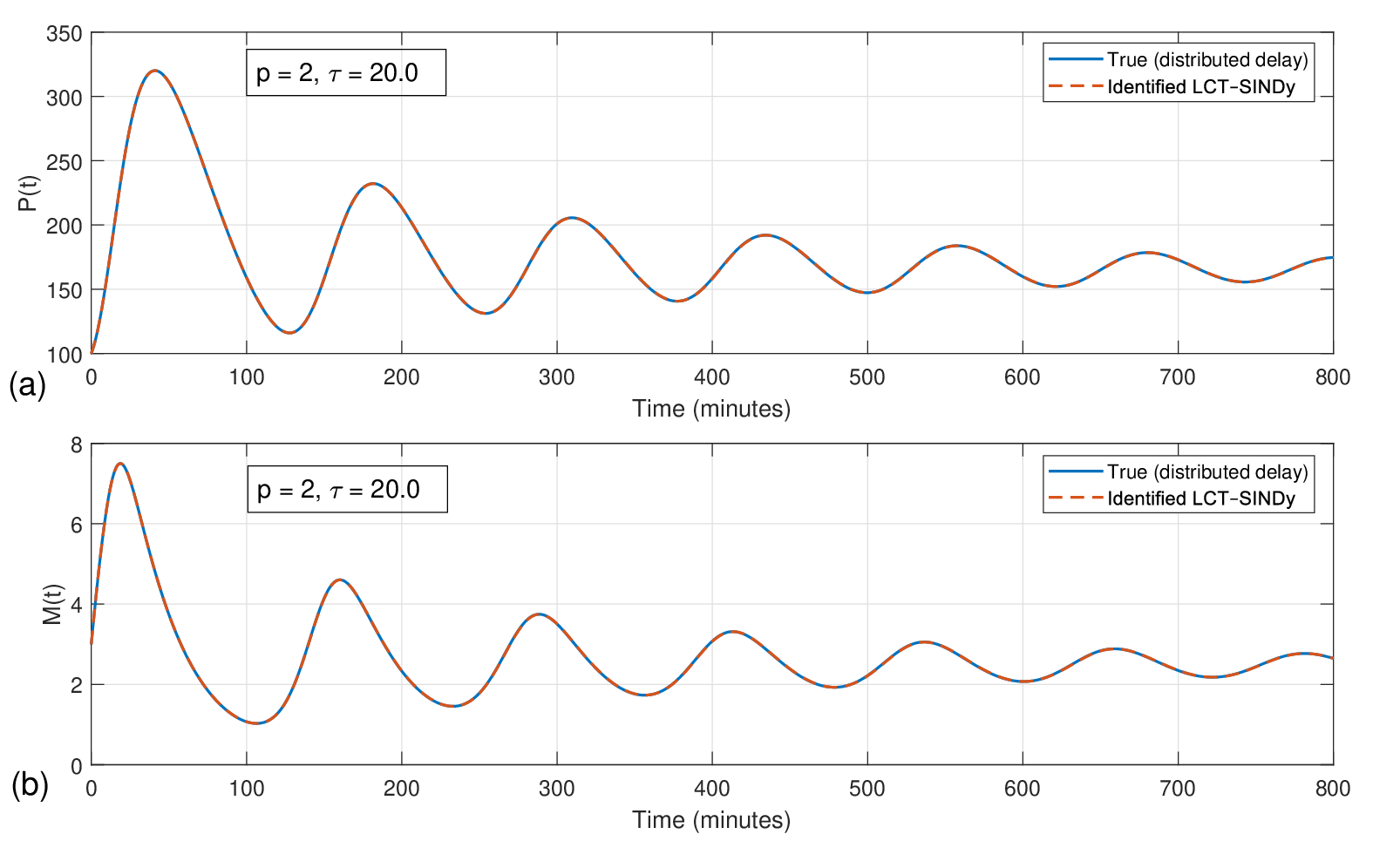}

\vspace{0.3cm}

\includegraphics[width=1.00\textwidth]{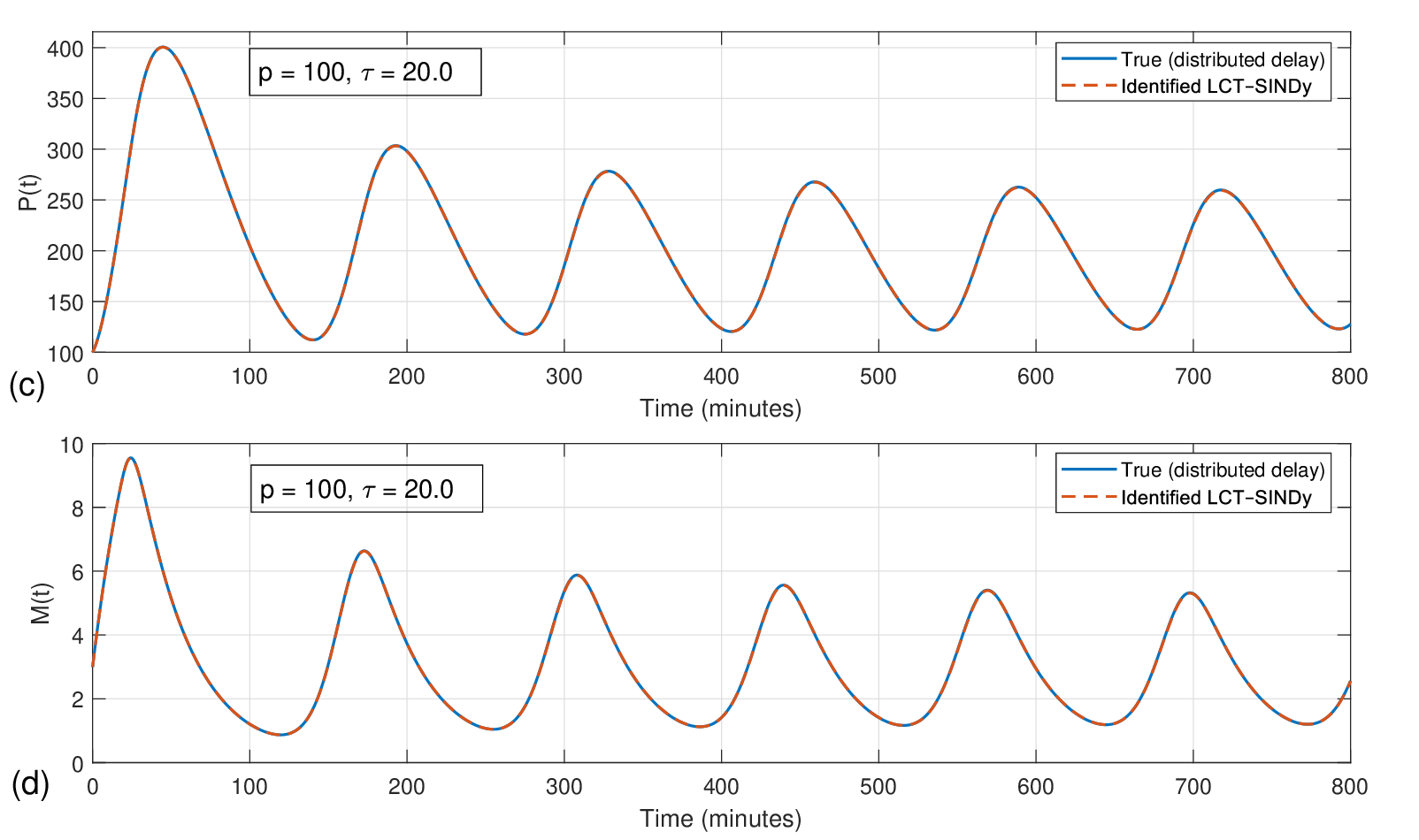}

\caption{$LCT-SIND^3y$ identification results for distributed-delay Hes1--mRNA dynamics.
Panels (a) and (b) correspond to a broadly distributed delay with $p = 2$,
while panels (c) and (d) correspond to a sharply distributed delay with $p = 100$.
Synthetic data are generated using the parameters listed in Table~\ref{tab:parameters}.}
\label{fig:LCT_SINDy_comparison}
\end{figure}

As expected, the discrete-delay SINDy method can identify the correct model (see Figure ~\ref{fig:discreteSINDy2}) only when the distributed delay is narrow, making the system effectively close to the discrete-delay case. However, when the delay is broadly distributed, the method fails to identify the correct model, as shown in Figure~\ref{fig:discreteSINDyH}. In Figure~\ref{fig:LCT_SINDy_comparison}, we apply $LCT-SIND^3y$ to two cases with true $p = 2$ and true $p = 100$. In both cases, the method successfully identifies the correct governing model, including the true delay $\tau$ and the nonlinear parameters $K$ and $n$, by selecting the best values from the candidate grid using the BIC.

\begin{table}[ht]
\centering
\caption{Identified model terms using $LCT-SIND^3y$. All polynomial terms of order $2 - 4$ were identified as zero and omitted for clarity. The ground-truth model has $p=2$ and $\tau=20$, and the identified mean delay is $\tau^\ast=20$ and $p^\ast=2$.}
\label{tab:LCT_SINDy_terms}
\vspace{-0.3cm}
\begin{tabular}{lccccccc}
\toprule
\textbf{Eq.} 
    & $1$ 
    & $M$ 
    & $P$ 
    & $\mathrm{Hill}(z_n)$
    & $\mathrm{RSS}$ 
    & $\mathrm{Rel.\ Error}$ 
    & $\mathrm{BIC}$ \\
\midrule
$\dot{M}$ 
    & $0$
    & $-0.029999$
    & $0$
    & $0.99997$
    & $1.61 \times 10^{-2}$
    & $1.537 \times 10^{-3}$
    & $-8.648 \times 10^{3}$ \\
$\dot{P}$
    & $0$
    & $1.9999$
    & $-0.029999$
    & $0$
    & $2.027 \times 10$
    & $8.765 \times 10^{-4}$
    & $-2.931 \times 10^{3}$ \\
\bottomrule
\end{tabular}

\begin{minipage}{0.95\linewidth}
\footnotesize
\vspace{.3cm}
Total model fit: BIC = $-1.1580 \times 10^{4}$, 
RSS = $2.028\times 10$.
\end{minipage}
\end{table}

\subsection{Effects of Sampling and Noise}
\label{sec4.2}
In the previous subsection, we demonstrated the importance of including 
distributed delays in the SINDy library, rather than relying solely on 
discrete delays, using clean, noise-free data. Although such data are not 
representative of experimental measurements, they allowed us to clearly isolate 
and illustrate the key point: when the true delayed influence is broadly 
distributed, a discrete-delay library is structurally insufficient, whereas an 
LCT-based distributed-delay library enables correct identification. In this subsection, we numerically investigate the effects of sampling and noise on LCT reconstruction.

Section~(\ref{sec:lct-sindy}) established that the Linear Chain Trick ODE system is mathematically equivalent to a distributed-delay 
formulation.Traditionally, LCT is introduced to numerically solve a known 
distributed-delay differential equation by converting the integral delay term into 
a system of coupled ODEs. In contrast, within the present SINDy framework, we 
assume that only the time-series data are available, and we employ the LCT as a data-driven mechanism to approximate the distributed delay. In this setting, 
the chain equations are solved using initial conditions derived directly from the 
observed data, while SINDy is responsible for identifying the remaining governing 
equations that best describe the system dynamics. Because the LCT chain requires a continuous input signal, the measured data, 
which are typically discrete, unevenly sampled, and affected by measurement noise, 
are first smoothed using a Savitzky–Golay (SG) filter with a third-order polynomial 
and an appropriate window length consistent with the sampling rate. The smoothed data points are then interpolated using a 
shape-preserving cubic interpolation (\texttt{'pchip'}) to construct a continuous 
function that drives the chain dynamics. This combination of smoothing and 
interpolation ensures that the ODE solver can evaluate the input at arbitrary time 
points during the integration of the LCT system while maintaining both continuity 
and robustness against measurement noise. However, when the input data are noisy or sparsely sampled, the accuracy of the 
reconstructed delayed state produced by the LCT depends critically on the quality 
of both the interpolation and the data itself. It is therefore essential to quantify 
how sampling sparsity and measurement noise affect the LCT reconstruction.

In this subsection, we investigate the robustness of the last chain state, $z_p$ which represents the distributed-delay contribution to degraded data quality. We compare the reconstructed $z_p^{\text{id}}$ with the ground-truth 
$z_p^{\text{true}}$ across a range of sampling intervals
\[
\Delta t \in \{1, 2, 3, 5, 8, 12, 15, 20\} \ \text{minutes},
\]
and additive Gaussian white noise levels
\[
\eta \in \{0.1, 0.2, \ldots, 1.0\}.
\]
For each configuration, we compute the relative error of the LCT output and, for 
comparison, the relative error in the numerical derivatives obtained from the same 
noisy and sparsely sampled data (see Fig.~\ref{fig:noise_and_sample}).

\begin{figure}[htbp]
    \centering
            \includegraphics[width=0.7\textwidth]{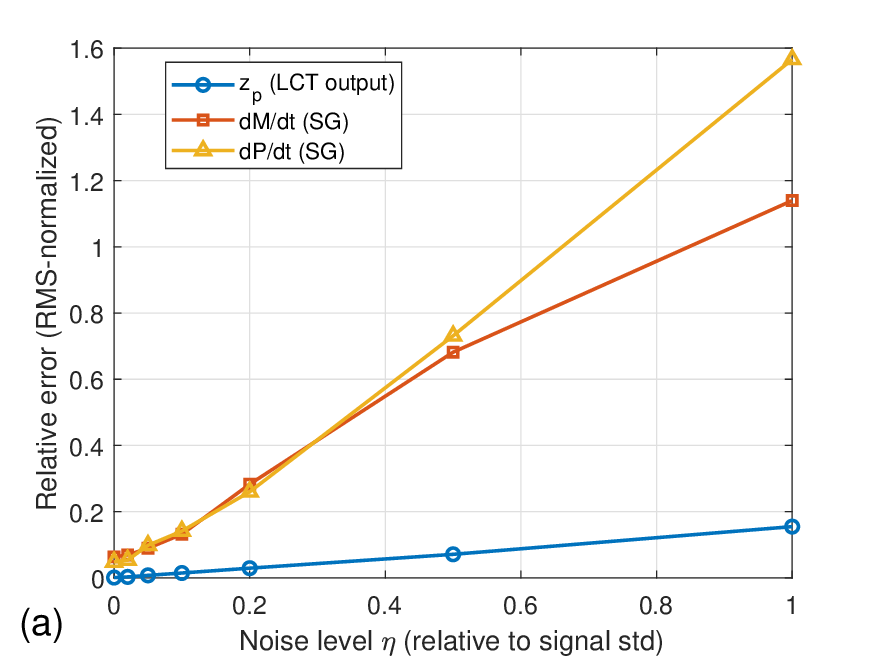}
               \label{fig:noiseFig}
          \includegraphics[width=.7\textwidth]{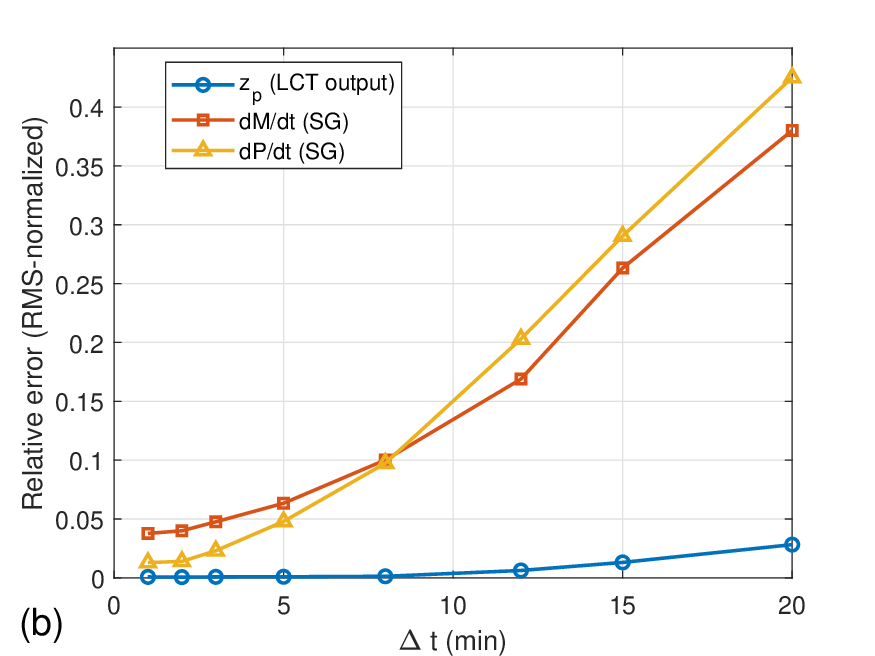}
                \label{fig:sampleFig}
    \caption{Comparison of LCT robustness. (a) Reconstruction stability vs noise ($\Delta t = 5$ minutes).  Relative error of the identified $z_p$ (blue) computed from data sampled every 5 minutes, plotted against noise level, along with derivative errors of $P$ (red) and $M$ (yellow). (b) Reconstruction stability vs sampling (noise = 0.00). Relative error of the identified $z_p$ obtained from noise-free data with varying sampling step sizes.}
    \label{fig:noise_and_sample}
\end{figure}

\begin{figure*}
  \centering
    \centering
    \includegraphics[width=0.48\linewidth]{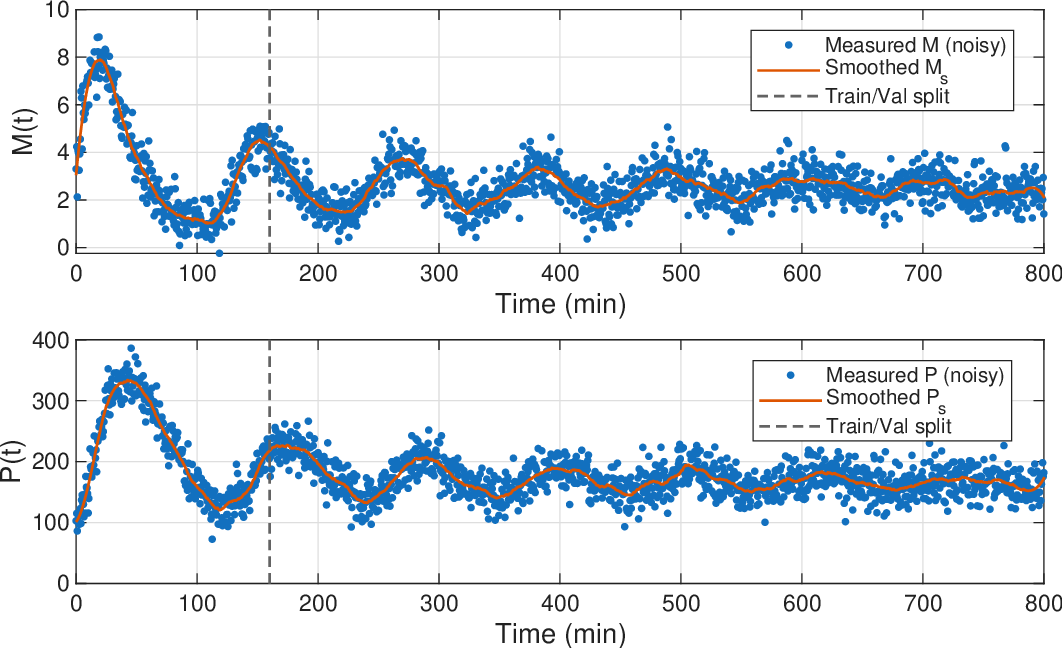} 
    \includegraphics[width=0.48\linewidth]{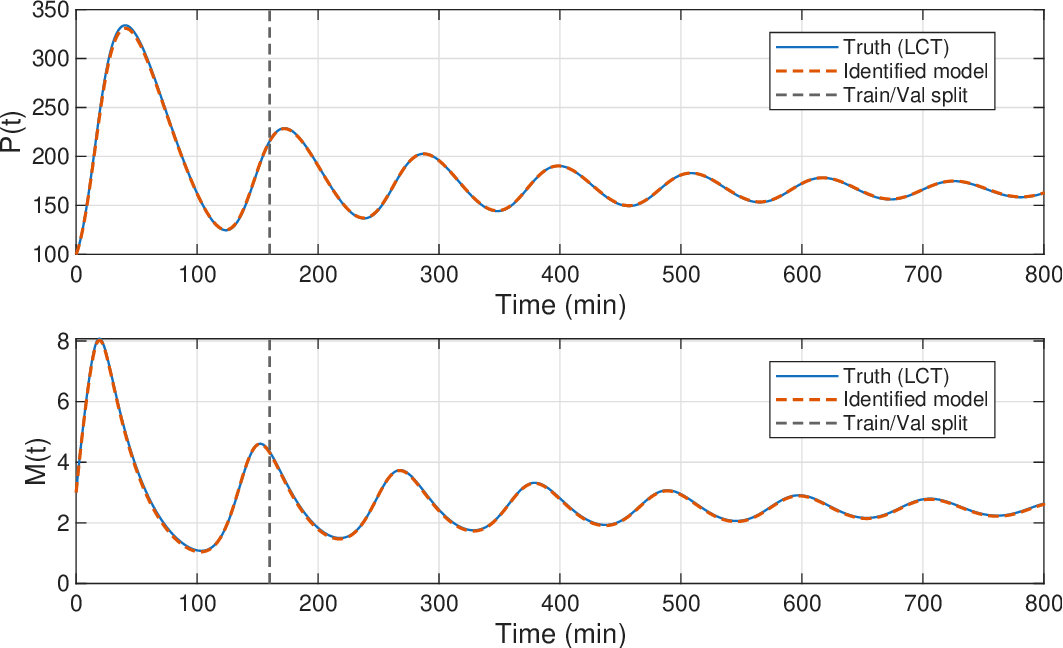}\\
    \includegraphics[width=0.48\linewidth]{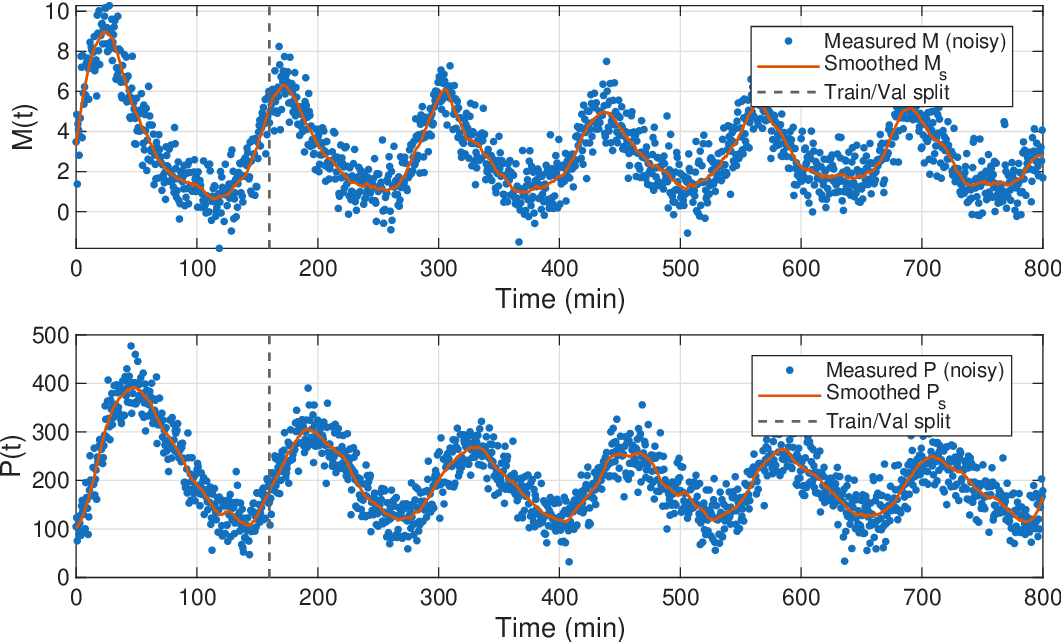}
    \includegraphics[width=0.48\linewidth]{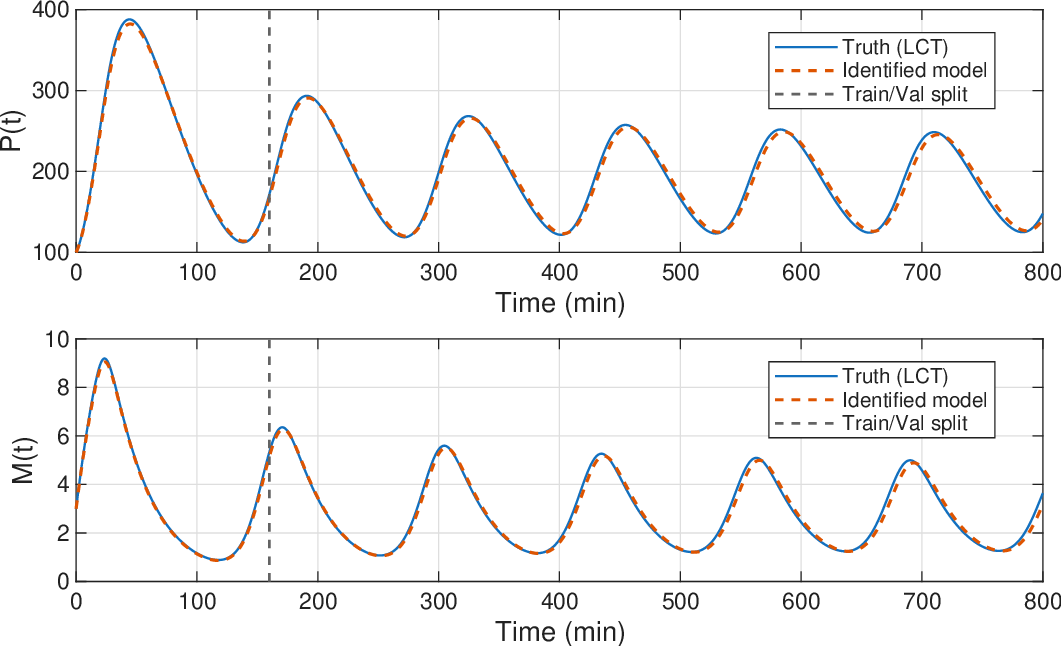}\\
\caption{LCT--SINDy identification results for noisy Hes1--mRNA data (noise intensity $0.5$) with sampling step $dt = 0.5\,\mathrm{min}$. Left column: noisy measurements (blue) together with the smoothed signals (red); the vertical dashed line indicates the split between training and validation data. Right column: comparison between the ground-truth trajectories (solid blue) and the trajectories generated by the identified LCT--SINDy model (dashed red). Top row: data generated with ground-truth parameters $\tau = 15$ and $n = 10$, correctly identified as $\tau^\ast = 15$ and $n^\ast = 10$. Bottom row: data generated with ground-truth parameters $\tau = 20$ and $n = 20$,
correctly identified as $\tau^\ast = 20$ and $n^\ast = 20$. Despite substantial measurement noise, the identified models accurately recover
both the delay and the chain length and reproduce the underlying dynamics.}

  \label{fig:lct-sindy-noisy-grid}
\end{figure*}
Our results show that the relative error of $z_p$ is consistently smaller than the 
corresponding derivative error across all sampling rates and noise intensities 
tested. This indicates that, in practical settings, whenever the experimental data 
are sufficiently informative to yield reasonable derivative estimates, they are 
also suitable for producing accurate LCT-based distributed-delay states. In other 
words, LCT is substantially more robust to sampling sparsity and measurement noise 
than numerical differentiation, which supports its use as a reliable mechanism for 
constructing delayed features in data-driven model discovery.\\

\subsection{An Example with Noisy Data}
\label{sec4.3}
In this subsection, we assess the performance of the $LCT$--$SIND^3y$ framework in the presence of measurement noise. To this end, we first solve the distributed-delay differential equation~\eqref{eq:hes1_distributed} to generate a reference trajectory, and then contaminate the simulated data with additive Gaussian noise, scaled relative to the standard deviation of the clean signal. The resulting noisy observations are sampled at intervals of $\Delta t = 0.5$ minutes with noise intensity $0.5$ (see Fig.~\ref{fig:lct-sindy-noisy-grid}).

As discussed in Section~\ref{Meth_cons}, prior to model identification the noisy data are denoised using the classical Savitzky--Golay filter in order to obtain a smooth approximation of the underlying trajectory. Sparse regression is then applied to the denoised data to identify the governing dynamics (see
Fig.~\ref{fig:lct-sindy-noisy-grid}). Despite the degraded data quality, the $LCT$--$SIND^3y$ framework successfully recovers the correct system structure,delay value, and chain length, demonstrating robustness to measurement noise.

In addition to the Hes1--mRNA regulatory system studied in Sections~4.1--4.3, we evaluated the $LCT$--$SIND^3y$ framework on two additional classes of delayed dynamical systems. In both cases, the method successfully recovered the correct model structure and delay, demonstrating that the proposed approach generalizes beyond the specific biological example considered earlier. Detailed numerical analyses of these two examples, including one that exhibits chaotic dynamics, are provided in the appendix; see Sections~\ref{Appendix} and Figures~\ref{S1},~\ref{fig:S2}.

\section{Discussion}
\label{Disc}
In recent years, the Sparse Identification of Nonlinear Dynamics (SINDy) framework has gained significant attention due to its success in discovering governing equations across various scientific and engineering disciplines. While early developments focused on ordinary and partial differential equations, recent extensions have enabled SINDy to identify systems with discrete delays. However, these approaches remain limited to models that explicitly include discrete delays in their libraries.

This study addresses a key gap by investigating the ability of SINDy to identify systems with distributed delays, which are common in biological and physical processes. In subsection~\ref{sec4.1}, we demonstrated that standard SINDy fails to recover broadly distributed delay effects when such dynamics are not represented in the candidate library. To overcome this limitation, we incorporated the LCT, which provides a natural representation for distributed delays and allows reconstruction of the delayed state variable directly from data. To the best of the authors’ knowledge, this work is the first to investigate the use of the LCT framework for recovering delayed state variables from noisy and sparsely sampled data. Moreover, we prove that the reconstructed delayed state is stable with respect to the input signal (see Proposition~\ref{p1}). In Section~\ref{sec4.2}, we analyzed how noise and sampling resolution affect the accuracy of the reconstructed delayed state. As shown in Fig.~\ref{fig:noise_and_sample}, the relative error between the identified and true delayed states increases with noise and poor sampling. Nevertheless, even under poor sampling and high noise levels, the reconstructed delayed state exhibits only small errors.

In Section~\ref{sec4.3}, we further evaluated the robustness of the $LCT-SIND^3y$ method using noisy training data. The results showed that the method successfully recovered the correct system structure and delay values for noise levels up to 0.5, maintaining close agreement with the true system dynamics. Despite these promising results, the accuracy of the method degrades once the
noise level or sampling sparsity exceeds the ranges explored here. This behavior is largely inherent to SINDy-type approaches, which rely on stable numerical
differentiation. In addition, we use the Bayesian information criterion (BIC) to select the model from a grid of candidate parameters, including $\tau$, $p$, and
the nonlinear coefficients. While BIC provides a principled trade-off between goodness of fit and model complexity, an exhaustive grid search becomes computationally expensive as the parameter space grows. Finally, the present linear chain formulation assumes an Erlang (gamma) kernel for the distributed delay. Although this family captures a broad class of unimodal delay distributions, it cannot represent arbitrary kernels (e.g., strongly multimodal or heavy-tailed delays), which may limit performance when the true memory
structure lies far outside the Erlang family.

Future work will focus on improving derivative estimation under high noise and sparse sampling conditions to further enhance the performance of the $LCT-SIND^3y$ framework. Additionally, we plan to develop more efficient and noise-resilient strategies for identifying delay values, chain lengths, and nonlinear parameters, reducing dependence on exhaustive grid searches and improving the reliability of model selection.
Overall, this study highlights the importance of extending SINDy to identify distributed delay systems and demonstrates that integrating the Linear Chain Trick provides a powerful and interpretable approach for handling noisy experimental data.\\

\section{Conclusion}
We introduced $LCT-SIND^3y$, a sparse identification framework for discovering nonlinear dynamical systems with distributed delays directly from the corresposnding time-series data. By incorporating the Linear Chain Trick into the SINDy methodology, distributed-delay effects are represented through a finite-dimensional ODE augmentation that preserves interpretability and enables standard sparse regression. Numerical experiments across multiple benchmark systems demonstrate that the proposed approach accurately recovers governing equations and delay characteristics under various sampling and noise conditions. These results show that $LCT-SIND^3y$ further extends data-driven model discovery to a broad class of systems with distributed-delay dynamics.

\section*{Code Availability}
\noindent All code used in this study is publicly available on GitHub at \url{https://github.com/m7eimmed/LCT-SINDy.git}.

\section*{Acknowledgment}
\noindent This study was supported by the National Science Foundation under Award number 2325267. Its contents are solely the responsibility of the authors and do not necessarily represent the official views of NSF.

\section*{Disclosure}
\noindent During the preparation of this work, the authors used Grammarly and ChatGPT to improve the writing style.  After using these tools, the authors reviewed and edited the content as necessary and took full responsibility for the publication's content.

\section*{Appendix}
\label{Appendix}
\setcounter{figure}{0}
\renewcommand{\thefigure}{S\arabic{figure}}
\paragraph{Additional Examples}
In addition to the Hes1--mRNA regulatory system studied in Sections~4.1--4.3, we evaluated the $LCT$--$SIND^3y$ framework on two additional classes of delayed dynamical systems. In both cases, the method successfully recovered the correct model structure and delay, demonstrating that the proposed approach generalizes beyond the specific biological example considered earlier.

\paragraph{Distributed-delay logistic growth model.}
As a first additional example, we considered the classical logistic growth equation with distributed-delay feedback,
\[
x'(t)
=
r\,x(t)\left(1 - \frac{1}{K}
\int_{0}^{\infty} K_{p,a}(s)\,x(t-s)\,ds \right),
\]
where $r>0$ is the intrinsic growth rate, $K>0$ is the carrying capacity, and $K_{p,a}$ is the Erlang kernel with mean delay $\tau = p/a$. Using only noisy observations of $x(t)$, the $LCT$--$SIND^3y$ procedure
accurately reconstructed the underlying distributed-delay dynamics, correctly identifying both the delay and the chain length from a grid of candidate values \( \tau \) and \( p \in \{1,2,3,\ldots,10\} \). The estimated parameters remained stable across a range of noise intensities (see Figure ~\ref{S1}).
\begin{figure}
  \centering
    \centering
    \includegraphics[width=0.48\linewidth]{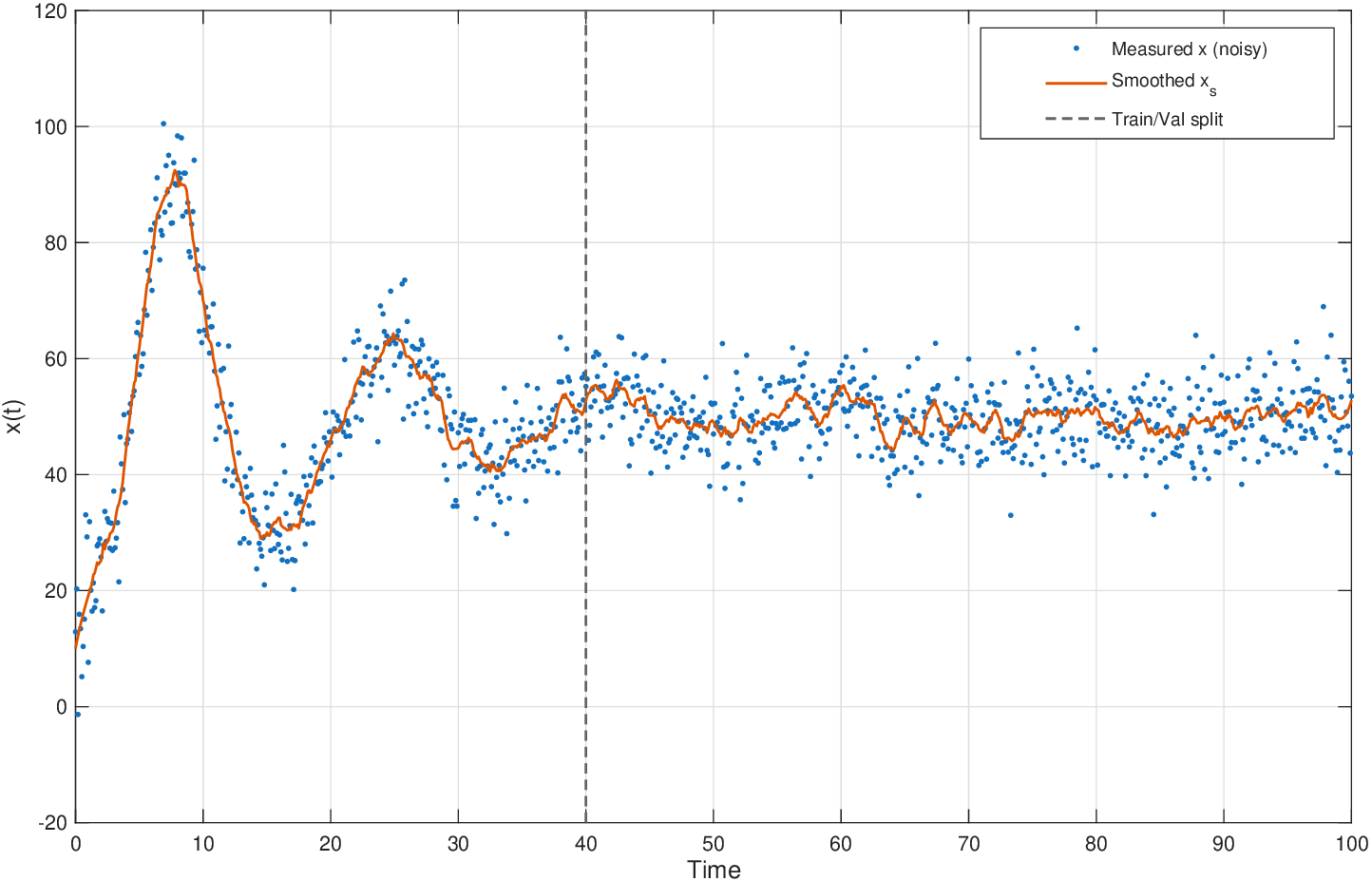} 
    \includegraphics[width=0.48\linewidth]{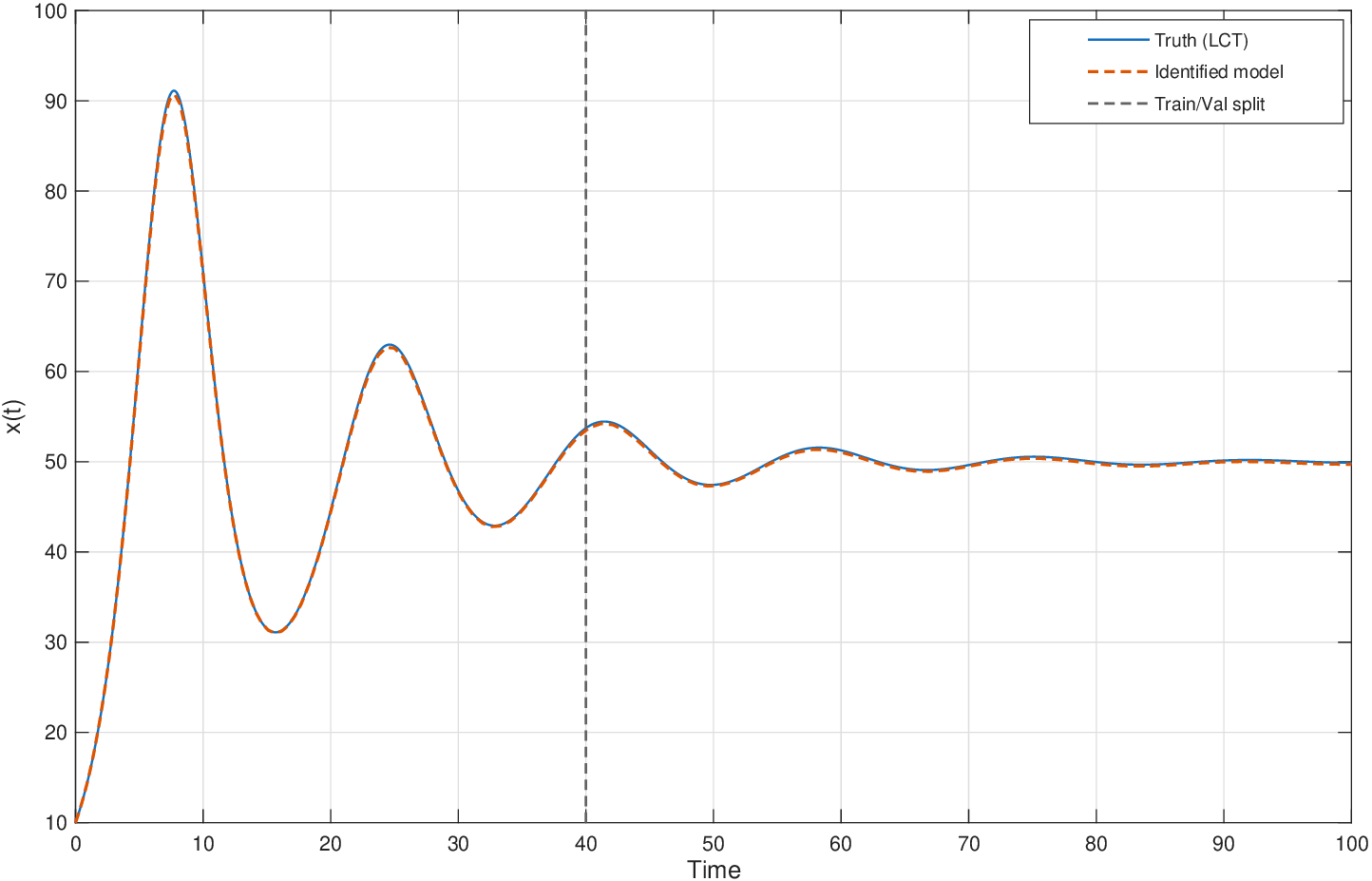}\\
    \includegraphics[width=0.48\linewidth]{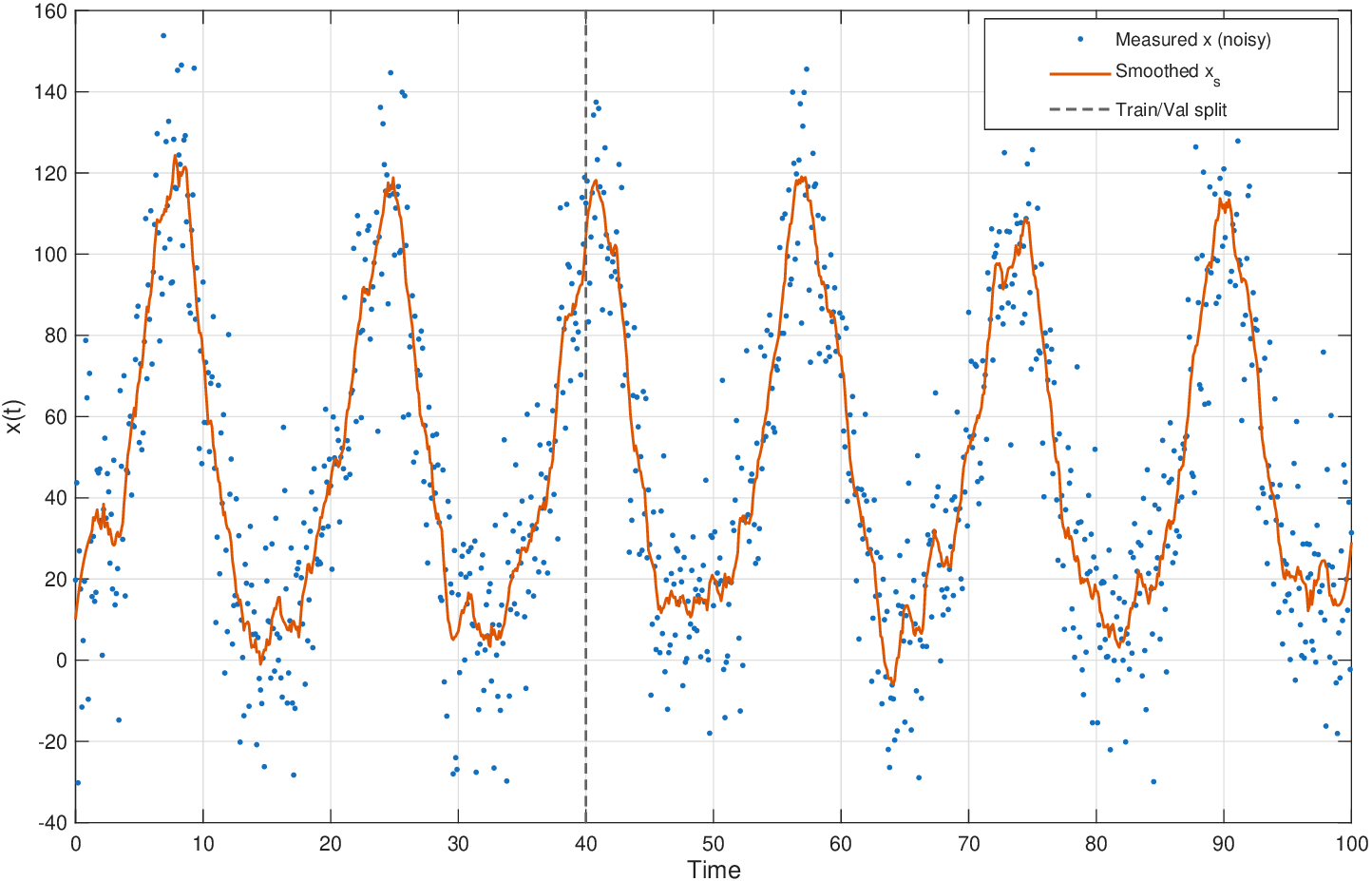}
    \includegraphics[width=0.48\linewidth]{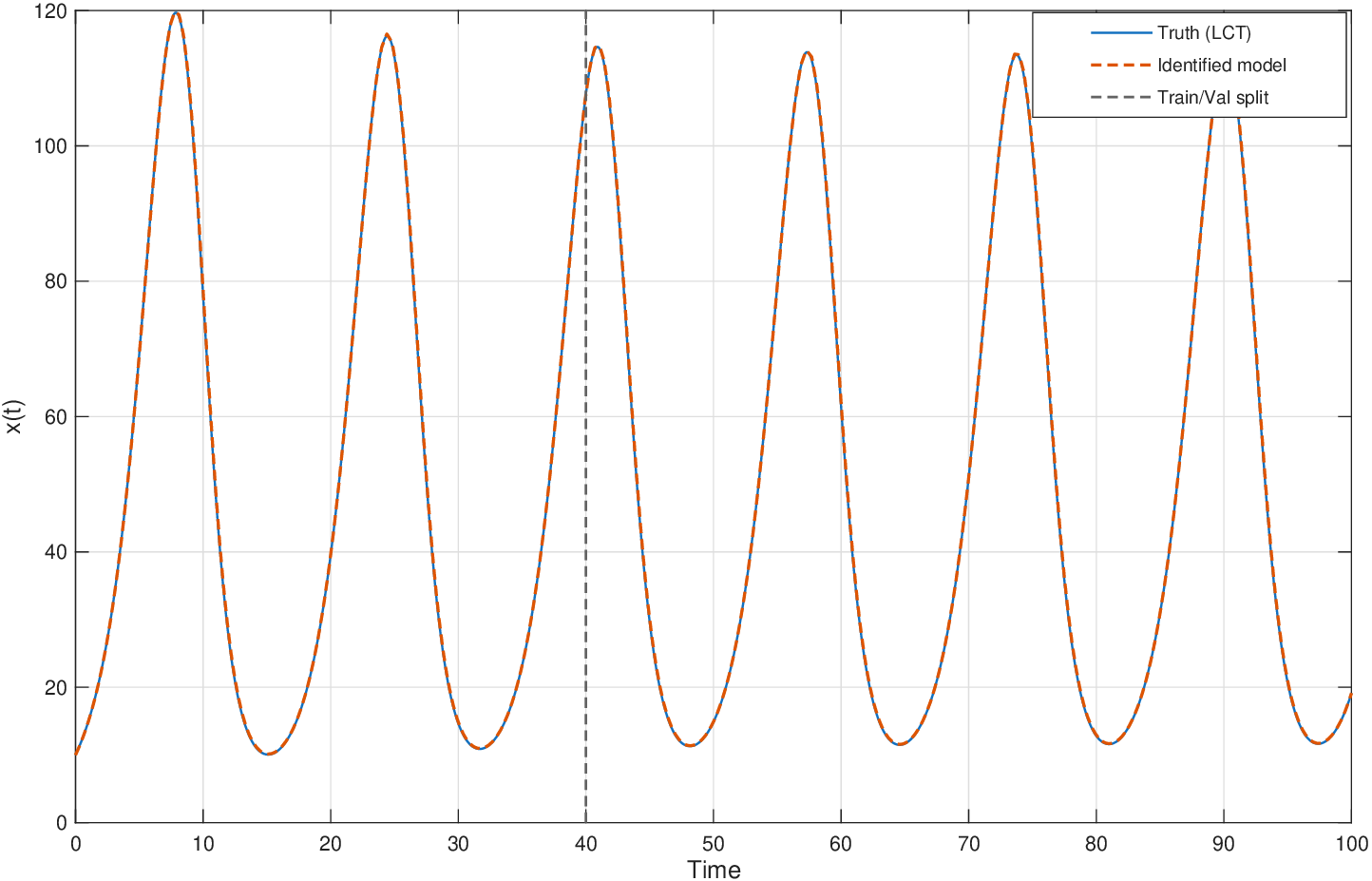}\\
\caption{LCT--SINDy identification results for noisy logistic model data} (noise intensity $0.5$) with sampling step $dt = 0.1$. Left column: noisy measurements (blue) together with the smoothed signals (red); the vertical dashed line indicates the split between training and validation data. Right column: comparison between the ground-truth trajectories (solid blue) and
the trajectories generated by the identified LCT--SINDy model (dashed red). Top row: data generated with a broad distributed delay ($n = 2$, $\tau = 4$), which stabilizes the dynamics. Bottom row: data generated with a narrower distributed delay ($n = 10$, $\tau = 4$), which induces oscillatory behavior. The identified models accurately capture the qualitative dynamical regimes associated with changes in the width of the distributed delay.
\label{S1}
\end{figure}

\paragraph{Ikeda delayed-feedback model.}
As a second benchmark, we considered the Ikeda model~\cite{IkedaMatsumoto1987}, a canonical example of a scalar delay differential equation exhibiting
delay-induced chaos, originally derived in the context of nonlinear optical systems with delayed feedback:
\[
x'(t) = -x(t) + \alpha \sin\bigl(x(t-\tau)\bigr),
\]
where $\alpha>0$ controls the feedback strength and $\tau>0$ is the delay. Despite its low-dimensional form, the presence of the delay renders the dynamics effectively infinite dimensional and highly sensitive to parameter perturbations.

Applying the proposed \(\mathrm{LCT}\text{--}\mathrm{SIND^3y}\) framework to observations of \(x(t)\), we performed a grid search over candidate delay values
\[
\tau \in \{1, 1.01, 1.02, \ldots, 2\}.
\]
The method successfully identified the correct delay and reconstructed the underlying nonlinear delayed interaction through its linear--chain approximation.
\begin{table}[ht]
\centering
\caption{Identified Ikeda models using LCT--SINDy under increasing noise levels.
All higher-order polynomial and mixed terms were identified as zero and omitted for clarity.
The ground-truth system is $\dot{x}(t) = -x(t) + 6\sin(x(t-\tau))$ with $\tau=1.59$.}
\label{tab:ikeda_noise_robustness}
\vspace{-0.3cm}
\begin{tabular}{c l c}
\toprule
\textbf{Noise level} 
& \textbf{Identified model} 
& \textbf{Identified delay $\tau^\ast$} \\
\midrule
$0$ 
& $\dot{x} = -1.0072\,x + 5.96005\,\sin(z)$ 
& $1.59$ \\

$0.02$ 
& $\dot{x} = -1.00519\,x + 5.86707\,\sin(z)$ 
& $1.59$ \\

$0.05$ 
& $\dot{x} = -0.996927\,x + 5.79499\,\sin(z)$ 
& $1.59$ \\

$0.10$ 
& $\dot{x} = -0.967207\,x + 5.60137\,\sin(z)$ 
& $1.59$ \\

$0.20$ 
& $\dot{x} = -0.874275\,x + 5.08357\,\sin(z)$ 
& $1.59$ \\
\bottomrule
\end{tabular}

\begin{minipage}{0.95\linewidth}
\footnotesize
\vspace{0.3cm}
 \( z \) denotes the terminal state of the linear chain and provides an approximation of the delayed state \( x(t-\tau) \).
\label{table6}
\end{minipage}
\end{table}

Across varying noise levels (Table~\ref{table6}), the framework consistently recovered the correct model structure, with only small errors in the estimated parameters. As the noise level increases, these errors grow modestly, primarily due to inaccuracies in numerical derivative estimation from denoised data. Because the underlying system is chaotic, even small errors in parameter estimates or derivatives can result in large deviations between the reconstructed and true trajectories, as shown in  Fig.~\ref{fig:S2}.
\begin{figure}
\centering
\includegraphics[width=1.00\textwidth]{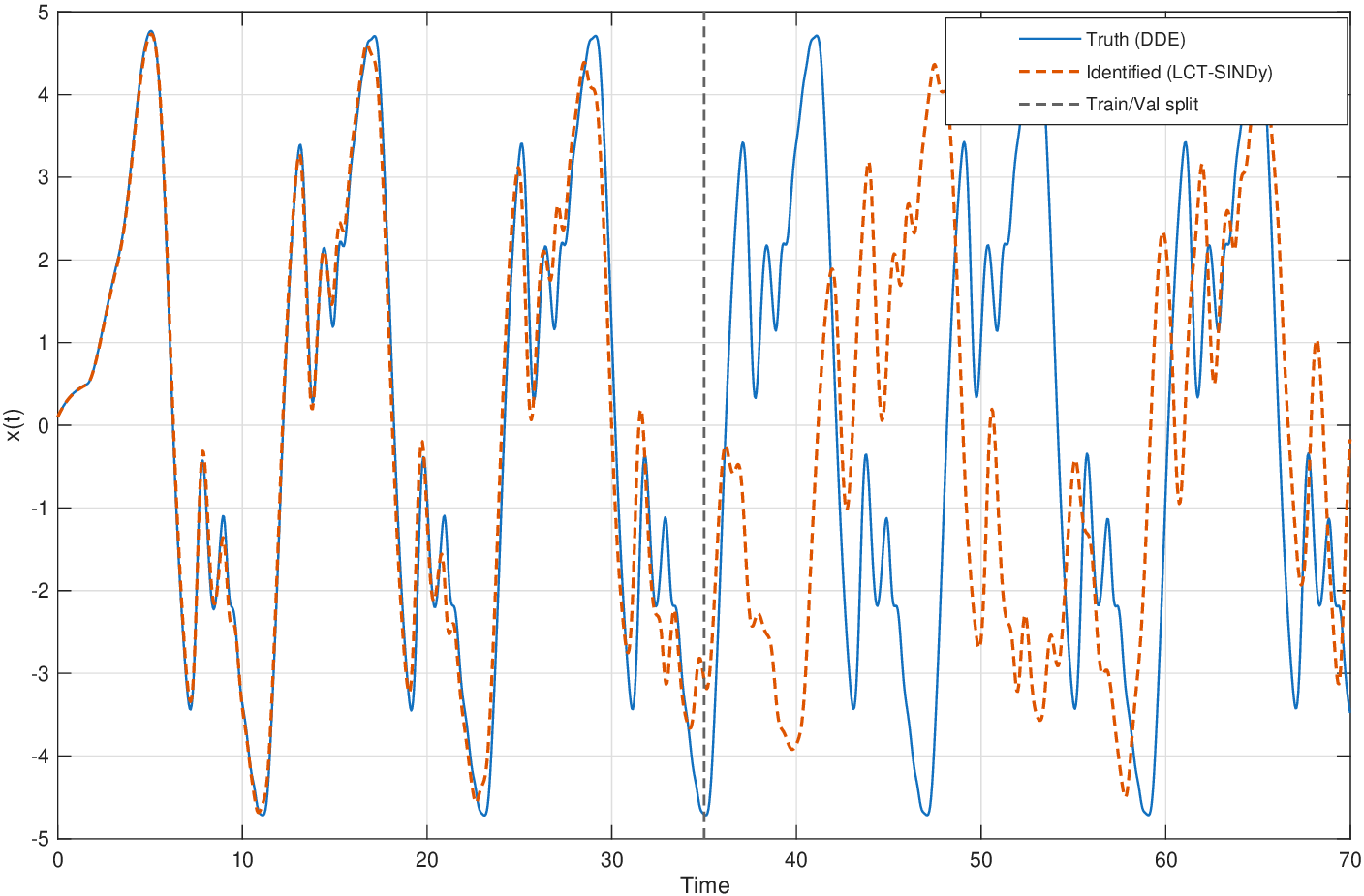} 

\caption{LCT--SINDy identification results for the Ikeda delay system with noise level $0$ and time step $\Delta t = 0.05$. The blue solid curve represents the ground-truth trajectory generated by the delay differential equation, while the red dashed curve corresponds to the trajectory produced by the identified LCT--SINDy model. The vertical dashed line indicates the split between training and validation data at $t = 35$; data on $[0,35]$ are used for model identification based on derivative information, and model performance is evaluated on the validation interval $[35,70]$.
Despite the chaotic nature of the system, the identified model accurately reconstructs the governing delayed nonlinear interaction.}

\label{fig:S2}
\end{figure}

For the selection criterion over the grid of candidate delay values, we use the derivative error rather than the trajectory error employed in previous examples. Specifically, the data are split into training and validation intervals. For each candidate delay, the sparse regression step is applied using the training data, yielding a candidate model associated with that delay value. Because the regression is performed on the training derivatives, models with incorrect delays may still fit the training derivative data by selecting appropriate library terms and coefficients. However, when these models are evaluated on the validation interval using the derivative error, candidates with incorrect delays exhibit larger errors, whereas the model corresponding to the true delay achieves the lowest validation error with the correct terms and coefficients. This selection criterion therefore enables the method to reliably identify the correct governing differential equation, even when the data are generated by a noisy chaotic system.


\end{document}